\UseRawInputEncoding
\documentclass[12pt]{amsart}
\usepackage{txfonts}      
\usepackage{amssymb}
\usepackage{eucal}
\usepackage{amsmath}
\usepackage{amscd}
\usepackage{xcolor}
\usepackage{multicol}
\usepackage[all]{xy}           
\usepackage{graphicx}
\usepackage{color}
\usepackage{colordvi}
\usepackage{xspace}
\usepackage{tikz}
\usepackage{makecell}
\usepackage{appendix}
\usepackage{amsthm}
\usepackage[misc]{ifsym}

\usepackage{ifpdf}
\ifpdf
\usepackage[colorlinks,final,backref=page,hyperindex]{hyperref}
\else
\usepackage[colorlinks,final,backref=page,hyperindex,hypertex]{hyperref}
\fi

\usepackage[active]{srcltx} 

\begin{document}
\newtheorem{theorem}{Theorem}[section]
\newtheorem{lemma}[theorem]{Lemma}
\newtheorem{corollary}[theorem]{Corollary}
\newtheorem{proposition}[theorem]{Proposition}
\theoremstyle{definition}
\newtheorem{definition}[theorem]{Definition}
\newtheorem{example}[theorem]{Example}
\newtheorem{remark}[theorem]{Remark}
\newtheorem{pdef}[theorem]{Proposition-Definition}
\newtheorem{condition}[theorem]{Condition}
\renewcommand{\labelenumi}{{\rm(\alph{enumi})}}
\renewcommand{\theenumi}{\alph{enumi}}
\baselineskip=14pt

\newcommand {\emptycomment}[1]{} 

\newcommand{\nc}{\newcommand}
\newcommand{\delete}[1]{}

\nc{\todo}[1]{\tred{To do:} #1}

\nc{\tred}[1]{\textcolor{red}{#1}}
\nc{\tblue}[1]{\textcolor{blue}{#1}}
\nc{\tgreen}[1]{\textcolor{green}{#1}}
\nc{\tpurple}[1]{\textcolor{purple}{#1}}
\nc{\tgray}[1]{\textcolor{gray}{#1}}
\nc{\torg}[1]{\textcolor{orange}{#1}}
\nc{\tmag}[1]{\textcolor{magenta}}
\nc{\btred}[1]{\textcolor{red}{\bf #1}}
\nc{\btblue}[1]{\textcolor{blue}{\bf #1}}
\nc{\btgreen}[1]{\textcolor{green}{\bf #1}}
\nc{\btpurple}[1]{\textcolor{purple}{\bf #1}}

	\nc{\mlabel}[1]{\label{#1}}  
	\nc{\mcite}[1]{\cite{#1}}  
	\nc{\mref}[1]{\ref{#1}}  
	\nc{\meqref}[1]{\eqref{#1}}  
	\nc{\mbibitem}[1]{\bibitem{#1}} 

\delete{
	\nc{\mlabel}[1]{\label{#1}  
		{ {\small\tgreen{\tt{{\ }(#1)}}}}}
	\nc{\mcite}[1]{\cite{#1}{\small{\tt{{\ }(#1)}}}}  
	\nc{\mref}[1]{\ref{#1}{\small{\tred{\tt{{\ }(#1)}}}}}  
	\nc{\meqref}[1]{\eqref{#1}{{\tt{{\ }(#1)}}}}  
	\nc{\mbibitem}[1]{\bibitem[\bf #1]{#1}} 
}

\nc{\cm}[1]{\textcolor{red}{Chengming:#1}}
\nc{\yy}[1]{\textcolor{blue}{Yanyong: #1}}
\nc{\zy}[1]{\textcolor{yellow}{Zhongyin: #1}}
\nc{\li}[1]{\textcolor{purple}{#1}}
\nc{\lir}[1]{\textcolor{purple}{Li:#1}}


\nc{\tforall}{\ \ \text{for all }}
\nc{\hatot}{\,\widehat{\otimes} \,}
\nc{\complete}{completed\xspace}
\nc{\wdhat}[1]{\widehat{#1}}

\nc{\ts}{\mathfrak{p}}
\nc{\mts}{c_{(i)}\ot d_{(j)}}

\nc{\NA}{{\bf NA}}
\nc{\LA}{{\bf Lie}}
\nc{\CLA}{{\bf CLA}}

\nc{\cybe}{CYBE\xspace}
\nc{\nybe}{NYBE\xspace}
\nc{\ccybe}{CCYBE\xspace}

\nc{\ndend}{pre-Novikov\xspace}
\nc{\calb}{\mathcal{B}}
\nc{\rk}{\mathrm{r}}
\newcommand{\g}{\mathfrak g}
\newcommand{\h}{\mathfrak h}
\newcommand{\pf}{\noindent{$Proof$.}\ }
\newcommand{\frkg}{\mathfrak g}
\newcommand{\frkh}{\mathfrak h}
\newcommand{\Id}{\rm{Id}}
\newcommand{\gl}{\mathfrak {gl}}
\newcommand{\ad}{\mathrm{ad}}
\newcommand{\add}{\frka\frkd}
\newcommand{\frka}{\mathfrak a}
\newcommand{\frkb}{\mathfrak b}
\newcommand{\frkc}{\mathfrak c}
\newcommand{\frkd}{\mathfrak d}
\newcommand {\comment}[1]{{\marginpar{*}\scriptsize\textbf{Comments:} #1}}


\nc{\disp}[1]{\displaystyle{#1}}
\nc{\bin}[2]{ (_{\stackrel{\scs{#1}}{\scs{#2}}})}  
\nc{\binc}[2]{ \left (\!\! \begin{array}{c} \scs{#1}\\
    \scs{#2} \end{array}\!\! \right )}  
\nc{\bincc}[2]{  \left ( {\scs{#1} \atop
    \vspace{-.5cm}\scs{#2}} \right )}  
\nc{\ot}{\otimes}
\nc{\sot}{{\scriptstyle{\ot}}}
\nc{\otm}{\overline{\ot}}
\nc{\ola}[1]{\stackrel{#1}{\la}}

\nc{\scs}[1]{\scriptstyle{#1}} \nc{\mrm}[1]{{\rm #1}}

\nc{\dirlim}{\displaystyle{\lim_{\longrightarrow}}\,}
\nc{\invlim}{\displaystyle{\lim_{\longleftarrow}}\,}

\nc{\bfk}{{\bf k}} \nc{\bfone}{{\bf 1}}
\nc{\rpr}{\circ}
\nc{\dpr}{{\tiny\diamond}}
\nc{\rprpm}{{\rpr}}

\nc{\mmbox}[1]{\mbox{\ #1\ }} \nc{\ann}{\mrm{ann}}
\nc{\Aut}{\mrm{Aut}} \nc{\can}{\mrm{can}}
\nc{\twoalg}{{two-sided algebra}\xspace}
\nc{\colim}{\mrm{colim}}
\nc{\Cont}{\mrm{Cont}} \nc{\rchar}{\mrm{char}}
\nc{\cok}{\mrm{coker}} \nc{\dtf}{{R-{\rm tf}}} \nc{\dtor}{{R-{\rm
tor}}}
\renewcommand{\det}{\mrm{det}}
\nc{\depth}{{\mrm d}}
\nc{\End}{\mrm{End}} \nc{\Ext}{\mrm{Ext}}
\nc{\Fil}{\mrm{Fil}} \nc{\Frob}{\mrm{Frob}} \nc{\Gal}{\mrm{Gal}}
\nc{\GL}{\mrm{GL}} \nc{\Hom}{\mrm{Hom}} \nc{\hsr}{\mrm{H}}
\nc{\hpol}{\mrm{HP}}  \nc{\id}{\mrm{id}} \nc{\im}{\mrm{im}}

\nc{\incl}{\mrm{incl}} \nc{\length}{\mrm{length}}
\nc{\LR}{\mrm{LR}} \nc{\mchar}{\rm char} \nc{\NC}{\mrm{NC}}
\nc{\mpart}{\mrm{part}} \nc{\pl}{\mrm{PL}}
\nc{\ql}{{\QQ_\ell}} \nc{\qp}{{\QQ_p}}
\nc{\rank}{\mrm{rank}} \nc{\rba}{\rm{RBA }} \nc{\rbas}{\rm{RBAs }}
\nc{\rbpl}{\mrm{RBPL}}
\nc{\rbw}{\rm{RBW }} \nc{\rbws}{\rm{RBWs }} \nc{\rcot}{\mrm{cot}}
\nc{\rest}{\rm{controlled}\xspace}
\nc{\rdef}{\mrm{def}} \nc{\rdiv}{{\rm div}} \nc{\rtf}{{\rm tf}}
\nc{\rtor}{{\rm tor}} \nc{\res}{\mrm{res}} \nc{\SL}{\mrm{SL}}
\nc{\Spec}{\mrm{Spec}} \nc{\tor}{\mrm{tor}} \nc{\Tr}{\mrm{Tr}}
\nc{\mtr}{\mrm{sk}}

\nc{\ab}{\mathbf{Ab}} \nc{\Alg}{\mathbf{Alg}}

\nc{\BA}{{\mathbb A}} \nc{\CC}{{\mathbb C}} \nc{\DD}{{\mathbb D}}
\nc{\EE}{{\mathbb E}} \nc{\FF}{{\mathbb F}} \nc{\GG}{{\mathbb G}}
\nc{\HH}{{\mathbb H}} \nc{\LL}{{\mathbb L}} \nc{\NN}{{\mathbb N}}
\nc{\QQ}{{\mathbb Q}} \nc{\RR}{{\mathbb R}} \nc{\BS}{{\mathbb{S}}} \nc{\TT}{{\mathbb T}}
\nc{\VV}{{\mathbb V}} \nc{\ZZ}{{\mathbb Z}}


\nc{\calao}{{\mathcal A}} \nc{\cala}{{\mathcal A}}
\nc{\calc}{{\mathcal C}} \nc{\cald}{{\mathcal D}}
\nc{\cale}{{\mathcal E}} \nc{\calf}{{\mathcal F}}
\nc{\calfr}{{{\mathcal F}^{\,r}}} \nc{\calfo}{{\mathcal F}^0}
\nc{\calfro}{{\mathcal F}^{\,r,0}} \nc{\oF}{\overline{F}}
\nc{\calg}{{\mathcal G}} \nc{\calh}{{\mathcal H}}
\nc{\cali}{{\mathcal I}} \nc{\calj}{{\mathcal J}}
\nc{\call}{{\mathcal L}} \nc{\calm}{{\mathcal M}}
\nc{\caln}{{\mathcal N}} \nc{\calo}{{\mathcal O}}
\nc{\calp}{{\mathcal P}} \nc{\calq}{{\mathcal Q}} \nc{\calr}{{\mathcal R}}
\nc{\calt}{{\mathcal T}} \nc{\caltr}{{\mathcal T}^{\,r}}
\nc{\calu}{{\mathcal U}} \nc{\calv}{{\mathcal V}}
\nc{\calw}{{\mathcal W}} \nc{\calx}{{\mathcal X}}
\nc{\CA}{\mathcal{A}}

\nc{\fraka}{{\mathfrak a}} \nc{\frakB}{{\mathfrak B}}
\nc{\frakb}{{\mathfrak b}} \nc{\frakd}{{\mathfrak d}}
\nc{\oD}{\overline{D}}
\nc{\frakF}{{\mathfrak F}} \nc{\frakg}{{\mathfrak g}}
\nc{\frakm}{{\mathfrak m}} \nc{\frakM}{{\mathfrak M}}
\nc{\frakMo}{{\mathfrak M}^0} \nc{\frakp}{{\mathfrak p}}
\nc{\frakS}{{\mathfrak S}} \nc{\frakSo}{{\mathfrak S}^0}
\nc{\fraks}{{\mathfrak s}} \nc{\os}{\overline{\fraks}}
\nc{\frakT}{{\mathfrak T}}
\nc{\oT}{\overline{T}}
\nc{\frakX}{{\mathfrak X}} \nc{\frakXo}{{\mathfrak X}^0}
\nc{\frakx}{{\mathbf x}}
\nc{\frakTx}{\frakT}      
\nc{\frakTa}{\frakT^a}        
\nc{\frakTxo}{\frakTx^0}   
\nc{\caltao}{\calt^{a,0}}   
\nc{\ox}{\overline{\frakx}} \nc{\fraky}{{\mathfrak y}}
\nc{\frakz}{{\mathfrak z}} \nc{\oX}{\overline{X}}

\font\cyr=wncyr10


\title{Algebraic constructions for left-symmetric conformal algebras}

\author{Zhongyin Xu}
\address{School of Mathematics, Hangzhou Normal University,
Hangzhou, 311121, China}
\email{Xzy@stu.hznu.edu.cn}

\author{Yanyong Hong}
\address{School of Mathematics, Hangzhou Normal University,
Hangzhou, 311121, China}
\email{yyhong@hznu.edu.cn}

\subjclass[2010]{17A30, 17A60, 17B60, 17B69, 17D25}
\keywords{left-symmetric conformal algebra, extending structures, unified product, bicrossed product, crossed product}
\footnote {
Corresponding author: Y. Hong (yyhong@hznu.edu.cn).
}
\begin{abstract}
Let $R$ be a left-symmetric conformal algebra and $Q$ be a $\mathbb{C}[\partial]$-module.
We introduce the notion of a unified product for left-symmetric conformal algebras and apply it to construct an object $\mathcal{H}^2_R(Q,R)$ to describe and classify all left-symmetric conformal algebra structures on the direct sum $E=R\oplus Q$ as a $\mathbb{C}[\partial]$-module such that $R$ is a subalgebra of $E$ up to isomorphism whose restriction on $R$ is the identity map. Moreover, we study $\mathcal{H}^2_R(Q,R)$ in detail when $Q$, $R$ are free as $\mathbb{C}[\partial]$-modules and $\text{rank}Q=1$. Some special products such as crossed product and bicrossed product are also investigated.

\end{abstract}

\maketitle
\section{Introduction}
The notion of a Lie conformal algebra introduced by V. Kac in \cite{K1,K2} is an axiomatic description of singular part of the operator product expansion of chiral fields in two-dimensional conformal field theory. It is a useful tool for studying vertex algebras \cite{K1} and infinite-dimensional Lie algebras satisfying the locality property \cite{K}. The structure theory \cite{DK}, representation theory \cite{CK, CK1} and cohomology theory \cite{BK} of finite Lie conformal algebras have been well developed. On the other hand, the notion of a left-symmetric conformal algebra was introduced in \cite{HL} to investigate whether there exist compatible left-symmetric algebra structures on formal distribution Lie algebras. Notice that the notion of a left-symmetric pseudoalgebra was introduced in \cite{W}. The conformal commutator of a left-symmetric conformal algebra is a Lie conformal algebra and finite left-symmetric conformal algebras which are free $\mathbb{C}[\partial]$-modules can naturally  provide the solutions of conformal Yang-Baxter equation and conformal $S$-equation  \cite{HB}. Moreover, the theory of left-symmetric conformal bialgebras was established in \cite{HL2}, compatible left-symmetric conformal algebra structures on the Lie conformal algebra $W(a,b)$ were investigated in  \cite{LHZZ, WH},  central extensions and simplicities of a class of left-symmetric conformal algebras were studied in \cite{XH}, and the general cohomology theory was presented in \cite{ZH}.

In this paper, we intend to study the following structure problem of left-symmetric conformal algebras:\\
\textbf{The $\mathbb{C}[\partial]$-split extending structures problem:}
{\it Given a left-symmetric conformal algebra $R$ and a $\mathbb{C}[\partial]$-module $Q$ . Set $E=R\oplus Q$ where the direct sum is the
sum of $\mathbb{C}[\partial]$-modules. Describe and classify all left-symmetric conformal algebra structures on $E$ such that $R$ is a subalgebra of $E$ up to isomorphism whose restriction on $R$ is the identity map.}\\
From the point of view of left-symmetric conformal algebras, this problem is natural and important, i.e. how to obtain a larger left-symmetric conformal algebra from a given one. Similar problems for groups, associative algebras, Hopf algebras, Lie algebras, Leibniz algebras, left-symmetric algebras, Lie-$2$ algebras, Lie conformal algebras and so on have been studied in \cite{A1, A3, A4, A2, A5, H2, TW, HS} respectively. It should be pointed out that this problem is hard when $R=\{0\}$, since it is equal to classifying all left-symmetric conformal algebras of a given rank. Notice that it is difficult to present a complete classification of torsion-free left-symmetric conformal algebras whose rank is $2$, according to the results given in \cite{LHZZ, WH}. Therefore, we always assume $R\neq0$ in this paper.

This problem contains many important problems  in the structure theory of left-symmetric conformal algebras. \delete{For instance, when $Q$ is
a left-symmetric conformal algebra, the problem that how to describe and classify all left-symmetric conformal algebra structures on $E$ such that $R$ and $Q$ are two subalgebras of $E$ up to isomorphism is a special case of the $\mathbb{C}[\partial]$-split extending structures
problem.} For example, it includes the following problem:\\
\textbf{The $\mathbb{C}[\partial]$-split extension problem:}
{\it Given two left-symmetric conformal algebras $R$ and $Q$. Describe and classify all $\mathbb{C}[\partial]$-split exact sequences of left-symmetric conformal algebras as follows up to equivalence:}
\begin{equation}
0\rightarrow R\stackrel{i}{\longrightarrow} E\stackrel{\pi}{\longrightarrow} Q\rightarrow 0.
\end{equation}
Notice that the $\mathbb{C}[\partial]$-split sequence means that $E\cong R\oplus Q$ as a $\mathbb{C}[\partial]$-module. If the $\lambda$-products on $R$ are trivial, this problem is equal to the $\mathbb{C}[\partial]$-split central extension problem of a left-symmetric conformal algebra $Q$. It is known from \cite{ZH} that all such $\mathbb{C}[\partial]$-split central extensions up to equivalence can be characterized by  the second cohomology group $H^2(Q,R)$.   Therefore the study of the $\mathbb{C}[\partial]$-split extending structures problem is meaningful and is useful for investigating the structure theory of left-symmetric conformal algebras. In this paper, we introduce
the notion of a unified product for left-symmetric conformal algebras and apply it to construct an object $\mathcal{H}^2_R(Q,R)$ to give a theoretical answer for the $\mathbb{C}[\partial]$-split extending structures problem. Moreover, we study $\mathcal{H}^2_R(Q,R)$ in detail when $R$ is free as a $\mathbb{C}[\partial]$-module and $Q$ is free of rank one as a $\mathbb{C}[\partial]$-module. Some special products such as crossed product and bicrossed product are also investigated. It should be pointed out that any $E$ in the $\mathbb{C}[\partial]$-split extension problem is isomorphic to a crossed product of $R$ and $Q$. We also construct an object $\mathcal{HC}^2(Q,R)$ to characterize all $E$ in the $\mathbb{C}[\partial]$-split extension problem.

This paper is organized as follows. In Section 2,  some related definitions and results of left-symmetric conformal algebras are recalled.
In Section 3, we introduce the notion of a unified product for left-symmetric conformal algebras and construct an object $\mathcal{H}^2_R(Q,R)$ to give a theoretical answer for the $\mathbb{C}[\partial]$-split extending structures problem.
In Section 4, we study the unified products when $R$ is a free $\mathbb{C}[\partial]$-module and $Q$ is a free $\mathbb{C}[\partial]$-module of rank 1 in detail.
In Section 5, some special cases of unified products such as crossed products and bicrossed products are
introduced. Some examples are presented in details.

Throughout this paper, we denote by $\mathbb{C}$ the set of complex numbers.
All vector spaces and tensor products are taken over the complex field $\mathbb{C}$. For any vector space $V$, we use $V[\lambda]$ to denote the set of polynomials of $\lambda$ with coefficients in $V$.
\section{Preliminaries}
In this section, we recall some basic definitions and facts about left-symmetric conformal algebras. These facts can be referred to \cite{HL, K1}.
\begin{definition}
A {\bf left-symmetric conformal algebra} $R$ is a $\mathbb{C}\left[\partial\right]$-module with a $\lambda$-product $\cdot_\lambda\cdot$
which is a $\mathbb{C}$-bilinear map from $R\times R \rightarrow R\left[\lambda\right]$, satisfying
\begin{align}
&\partial a_\lambda b=-\lambda a_\lambda b,~~~~~~~~a_\lambda\partial b=(\partial+\lambda)a_\lambda b,\tag{conformal sesquilinearity}\label{c1}\\
&(a_\lambda b)_{\lambda+\mu}c-a_\lambda(b_\mu c)=(b_\mu a)_{\lambda+\mu}c-b_\mu(a_\lambda c),\tag{left-symmetry}\label{c2}
\end{align}
for $a,b,c\in R$. We denote it by $(R, \cdot_\lambda \cdot)$ or $R$.

A {\bf Lie conformal algebra} $R$ is a
$\mathbb{C}\left[\partial\right]$-module with a $\lambda$-bracket $[\cdot_\lambda\cdot]$
which is a $\mathbb{C}$-bilinear map from $R\times R \rightarrow R\left[\lambda\right]$, satisfying
\begin{align}
&[\partial a_\lambda b]=-\lambda [a_\lambda b],~~~~~~~~[a_\lambda\partial b]=(\partial+\lambda)[a_\lambda b],\tag{conformal sesquilinearity}\label{Lie1}\\
&[a_\lambda[b_\mu c]]=[[a_\lambda b]_{\lambda+\mu}c]-[b_\mu[a_\lambda c]],\;\;\;a, b, c\in R.\tag{Jacobi identity}\label{Lie2}
\end{align}

\delete{A {\bf Novikov conformal algebra} $R$ is a left-symmetric conformal algebra satisfying
\begin{equation}
(a_\lambda b)_{\lambda+\mu}c=(a_\lambda c)_{-\mu-\partial}b,\label{c3}
\end{equation}
for all $a,b,c \in R$.}
\end{definition}

\delete{A left-symmetric conformal algebra or Lie conformal algebra is said to be {\bf finite}, if it is finitely generated as a $\mathbb{C}\left[\partial\right]$-module. Otherwise, we call
it {\bf infinite}.}

\begin{example}\label{curl}
Let $(L,\circ) $ be a left-symmetric algebra. Then there is a natural left-symmetric conformal algebra structure on $\text{Cur} L=\mathbb{C}[\partial]\otimes L$  with
the $\lambda$-products
$$a_\lambda b=a\circ b, ~~~\text{$a$, $b\in L$.}$$

\end{example}
\begin{proposition} \cite[Theorem 3.2]{HL}\label{proplsca}
Let $R = \mathbb{C}[\partial]x$ be a left-symmetric conformal algebra which is free
      and of rank 1 as a $\mathbb{C}[\partial]$-module. Then $R$ is isomorphic to the left-symmetric conformal algebra with $\lambda$-product which is one of
three cases as follows:\\
(i) $x_\lambda x = 0$;\\
(ii) $x_\lambda x = x$;\\
(iii) $x_\lambda x = (\partial + \lambda + c)x$, for any $c \in \mathbb{C}$.
\end{proposition}

\begin{proposition}\cite[Proposition 2.5]{HL}
Let $(R, \cdot_\lambda \cdot)$ be a left-symmetric conformal algebra. Then we can define a Lie conformal algebra structure on $R$  with the following $\lambda$-brackets
\begin{equation}
[a_\lambda b]=a_\lambda b-b_{-\lambda-\partial}a,\;\;\;\text{$a$, $b\in R$.}
\end{equation}
Denote this Lie conformal algebra by $\mathfrak{g}(R)$, which is called the {\bf sub-adjacent} Lie conformal algebra of $R$ and $R$
is a {\bf compatible} left-symmetric conformal algebra structure on the Lie conformal algebra $\mathfrak{g}(R)$.
\end{proposition}
In what follows, we recall the definition of a bimodule over a left-symmetric conformal algebra.
\begin{definition}
Let $R$ be a left-symmetric conformal algebra and $V$ be a $\mathbb{C}[\partial]$-module. $V$ is called a {\bf bimodule of $R$} (or an {\bf $R$-bimodule})
if there are two $\mathbb{C}$-linear maps $R\otimes V\rightarrow V[\lambda]$, $a\otimes v\rightarrow a_\lambda v$ and
$V\otimes R\rightarrow V[\lambda]$, $v\otimes a\rightarrow v_\lambda a$ such that
\begin{align}
&(\partial a)_\lambda v=-\lambda a_\lambda v,~~~~(\partial v)_\lambda a=-\lambda v_\lambda a,\label{M1}\\
&a_\lambda(\partial v)=(\lambda+\partial)a_\lambda v,~~v_\lambda(\partial a)=(\lambda+\partial)v_\lambda a,\label{M2}\\
&(a_\lambda b)_{\lambda+\mu}v-a_\lambda(b_\mu v)=(b_\mu a)_{\lambda+\mu}v-b_\mu(a_\lambda v),\\
&(a_\lambda v)_{\lambda+\mu}b-a_\lambda(v_\mu b)=(v_\mu a)_{\lambda+\mu}b-v_\mu(a_\lambda b),
\end{align}
hold for all $a,b\in R$ and $v\in V$.
\end{definition}

\begin{definition}\label{def2.6}
Let $U$ and $V$ be two $\mathbb{C}[\partial]$-modules.
A {\bf left conformal linear map} from $U$ to $V$ is a $\mathbb{C}$-linear map $\varphi$: $U\rightarrow V[\lambda]$, denoted by $\varphi_\lambda$ such that $\varphi_\lambda(\partial a)=-\lambda \varphi_\lambda a$ for all $a\in U$.
A {\bf right conformal linear map} from $U$ to $V$ is a $\mathbb{C}$-linear map $\psi$: $U\rightarrow V[\lambda]$, denoted by $\psi_\lambda$ such that $\psi_\lambda(\partial a)=(\partial+\lambda)\psi_\lambda a$ for all $a\in U$.
A right conformal linear map is often shortly called as {\bf conformal linear map}.

In addition, let $W$ also be a
$\mathbb{C}[\partial]$-module. A {\bf conformal bilinear map} from $U\times V\rightarrow W$ is a $\mathbb{C}$-bilinear map $f:U\times V\rightarrow W[\lambda] $, denoted by $f_\lambda(\cdot,\cdot)$ such that
$f_\lambda(\partial a,b)=-\lambda f_\lambda (a,b)$ and $f_\lambda( a,\partial b)=(\lambda+\partial) f_\lambda (a,b)$ for all $a\in U$ and $b\in V$.
\end{definition}

Denote the $\mathbb{C}$-vector space of all conformal linear maps from $V$ to $V$ by $Cend(V)$. It has a canonical $\mathbb{C}[\partial] $-module structure given as
\begin{eqnarray}
(\partial \varphi)_\lambda=-\lambda\varphi_\lambda,\;\;\;\; \varphi\in Cend(V).
\end{eqnarray}
\begin{remark}
Let $R$ be a left-symmetric conformal algebra and $V$ be a $\mathbb{C}[\partial]$-module. Let
$l$ and $r$ be two $\mathbb{C}[\partial] $-module homomorphisms: $R\rightarrow Cend(V)$. Define $a_\lambda v={l(a)}_\lambda v$ and $v_\lambda a={r(a)}_{-\lambda-\partial}v$ for all $a\in R$ and $v\in V$. Then it is easy to see that $V$ is an $R$-bimodule if and only if the following conditions hold:
\begin{align}
&l(a_\lambda b)_{\lambda+\mu}v-l(a)_\lambda(l(b)_\mu v)=l(b_\mu a)_{\lambda+\mu}v-l(b)_\mu(l(a)_\lambda v),\label{bm1}\\
&r(b)_{-\lambda-\mu-\partial}(l(a)_\lambda v)-l(a)_\lambda(r(b)_{-\mu-\partial} v)
=r(b)_{-\lambda-\mu-\partial}(r(a)_{-\mu-\partial} v)-r(a_\lambda b)_{-\mu-\partial}v,\label{bm2}
\end{align}
for all $a,b\in R$ and $v\in V$. Therefore, we also denote this bimodule by $(V,l,r)$.
\end{remark}

\begin{definition}
Let $R$ be a left-symmetric conformal algebra. A conformal linear map $T_\lambda $: $R\rightarrow R[\lambda]$ is called a {\bf conformal semi-quasicentroid} if $T$
satisfies
\begin{equation}
T_{-\lambda-\mu-\partial}(a_\lambda b-b_\mu a)=a_\lambda T_{-\mu-\partial}(b)-b_\mu T_{-\lambda-\partial}(a),\;\;\;\;a, b\in R.
\end{equation}
\end{definition}
\begin{remark}
For any $b\in R$, there is a conformal semi-quasicentroid $T^b_{\lambda}$ associated to it defined by $T^b_{\lambda}(a)= a_{-\lambda-\partial} b$  for all $a\in R$. We call $T^b_{\lambda}$ an {\bf inner conformal semi-quasicentroid} of $R$ and denote by $CSQInn(R)$ the vector space of all inner conformal semi-quasicentroids of $R$.
\end{remark}

\begin{definition}
Let $R$ be a left-symmetric conformal algebra.  A {\bf twisted conformal derivation} of $R$ is $(D_\lambda,g_\lambda(\cdot,\partial))$ where
$ g_\lambda(\cdot,\partial)$ : $R\rightarrow \mathbb{C}[\lambda,\partial]$ is a left conformal linear map and
$D_\lambda$ : $R\rightarrow R[\lambda] $ is a conformal linear map satisfying for all $a,b\in R$
\begin{eqnarray*}
D_\lambda(a)_{\lambda+\mu}b+g_{-\lambda-\partial}(a,-\lambda-\mu)D_{\lambda+\mu}(b)=D_\lambda(a_\mu b)-a_\mu(D_\lambda(b)).
\end{eqnarray*}
In particular, for a twisted conformal derivation $(D_\lambda,g_\lambda(\cdot,\partial))$,  if $ g_\lambda(\cdot,\partial)$ is trivial, $D_\lambda$ is called a {\bf conformal derivation} of $R$.
\end{definition}
\begin{proposition}
Let $R=\mathbb{C}[\partial]L$ be a left-symmetric conformal algebra  with the $\lambda$-product defined by $[L_\lambda L]=(\lambda+\partial+c)L$, $c\in \mathbb{C}$. Then all conformal derivations of $R$ are zero.
\end{proposition}
\begin{proof}
Let $D_\lambda$ be a conformal derivation of $R$.
Set $D_\lambda L=a(\lambda,\partial)L$ for some $a(\lambda,\partial)\in \mathbb{C}[\lambda,\partial]$. Then
\begin{eqnarray*}
D_\lambda(L)_{\lambda+\mu}L=D_\lambda(L_\mu L)-L_\mu(D_\lambda(L))
\end{eqnarray*}
is equivalent to
\begin{eqnarray}\label{eql}
a(\lambda,-\lambda-\mu)(\lambda+\mu+\partial+c)=(\lambda+\mu+\partial+c)a(\lambda,\partial)-a(\lambda,\mu+\partial)(\mu+\partial+c).
\end{eqnarray}
Let $a(\lambda,\partial)=\sum^n_{i=0}a_i(\lambda)\partial^i$ with $a_n(\lambda)\neq0$. Then assuming $n>1$, if we equate terms of degree $n$ in $\partial$, we obtain
\begin{eqnarray*}
(\lambda-n\mu)a_n(\lambda)=0,
\end{eqnarray*}
getting a contradiction. So $a(\lambda,\partial)=a_1(\lambda)\partial+a_0(\lambda)$. Then \eqref{eql} becomes
\begin{eqnarray} \label{eql2}
a_1(\lambda)\lambda(-\lambda-2\mu-2\partial-c)=-a_0(\lambda)(\mu+\partial+c).
\end{eqnarray}
Therefore, $a_1(\lambda)=a_0(\lambda)=0$. Consequently, all conformal derivations of $R$ are zero.
\end{proof}
\begin{definition}\label{def2}
 Let $R$ be a left-symmetric conformal algebra, $Q$ a $\mathbb{C}[\partial]$-module and $E=R\oplus Q$ where the direct sum is the
 sum of $\mathbb{C}[\partial]$-modules. For a $\mathbb{C}[\partial]$-module homomorphism $\phi$ : $E\rightarrow E$, we consider the following
 diagram:
\begin{displaymath}
\xymatrix{
  & R \ar[d]^{Id} \ar[r]^i  & E \ar[d]^\phi \ar[r]^\pi &Q\ar[d]^{Id}            \\
  & R \ar[r]^i & E \ar[r]^\pi &Q                      }
\end{displaymath}
where $\pi$ : $E\rightarrow Q$ is the canonical projection of $E=R\oplus Q$ onto $Q$ and $i$ : $R\rightarrow E$ is the inclusion map.
A $\mathbb{C}[\partial]$-module homomorphism $\phi$ : $E\rightarrow E$ {\bf stabilizes $R$} (resp. {\bf co-stabilizes $Q$}) if the left square
(resp. the right square) of the above diagram is commutative.
 Let $\cdot_\lambda\cdot$ and $\circ_\lambda$ be two left-symmetric  conformal algebra structures on $E$ both containing $R$ as a
 left-symmetric conformal subalgebra. $(E,\cdot_\lambda\cdot)$ and $(E,\circ_\lambda) $ are called {\bf equivalent} denoted by
 $(E,\cdot_\lambda\cdot)\equiv(E,\circ_\lambda) $, if there exists a left-symmetric conformal algebra isomorphism
 $\phi$ : $(E,\cdot_\lambda\cdot)\rightarrow(E,\circ_\lambda) $ stabilizing $R$.\par
 If there exists a left-symmetric conformal algebra isomorphism $\phi$ : $(E,\cdot_\lambda\cdot)\rightarrow(E,\circ_\lambda) $
 which stabilizes $R$ and co-stabilizes $Q$, then $(E,\cdot_\lambda\cdot)$ and $(E,\circ_\lambda)$ are called {\bf cohomologous}, which is denoted
 by $(E,\cdot_\lambda\cdot)\approx(E,\circ_\lambda)  $.

It is not hard to see that `` $\equiv$ " and `` $\approx$  " are equivalence relations on the set of all left-symmetric
conformal algebra structures on $E$ containing $R$ as a left-symmetric conformal subalgebra and we denote the set of all equivalence
classes via `` $\equiv$ " and `` $\approx$  " by $CExtd(E,R)$ and $CExtd'(E,R)$ respectively. Therefore, $CExtd(E,R)$ is the classifying object of the $\mathbb{C}[\partial]$-split extending structures problem and $CExtd'(E,R)$ gives a classification of all left-symmetric conformal algebra structures on $E=R\oplus Q$ containing $R$ as a subalgebra up to isomorphism which stabilizes $R$ and co-stabilizes $Q$.
\end{definition}

\section{Unified products for left-symmetric conformal algebras}
In this section, we will introduce the notion of a unified product for left-symmetric conformal algebras and apply it to construct an object to
give a theoretical answer for the $\mathbb{C}[\partial]$-split extending structures problem.

\begin{definition}
Let $R$ be a left-symmetric conformal algebra and $Q$ a $\mathbb{C}[\partial]$-module. An {\bf extending datum} of $R$ by $Q$ is a system
$\Omega(R,Q)=(\varphi,\psi,l,r,g_\lambda(\cdot,\cdot), \circ_\lambda)$ consisting of four $\mathbb{C}[\partial]$-module homomorphisms and two conformal bilinear maps as follows:
$$
\begin{aligned}
&l,r:R\rightarrow Cend(Q),~~~~~~~~~\varphi,\psi:Q\rightarrow Cend(R),\\
&g_\lambda(\cdot,\cdot):Q\times Q\rightarrow R[\lambda],~~~~~~~~~\circ_\lambda:Q\times Q\rightarrow Q[\lambda].
\end{aligned}
$$
Let $\Omega(R,Q)=(\varphi,\psi,l,r,g_\lambda(\cdot,\cdot), \circ_\lambda)$ be an extending datum.
We denote by $R\natural _{\Omega(R,Q)}Q=R\natural Q$ the $\mathbb{C}[\partial]$-module $R\oplus Q $ with the natural $\mathbb{C}[\partial]$-module action:
$\partial(a+x)=\partial a+\partial x$ and the bilinear map $\cdot_\lambda \cdot : (R\oplus Q )\times(R\oplus Q)\rightarrow (R\oplus Q)[\lambda]$
defined by
\begin{equation}
(a+x)_\lambda(b+y)\\
=(a_\lambda b+\varphi (x)_\lambda b+\psi(y)_{-\lambda-\partial}a+g_\lambda(x,y))+(x\circ_\lambda y+l(a)_\lambda y+r(b)_{-\lambda-\partial} x)\label{lambdap}
\end{equation}
for all $a,b\in R,x,y\in Q$. $R\natural Q$ is called the {\bf unified product} of $R$ and $\Omega(R,Q)$ if it is a left-symmetric conformal algebra
with the $ \lambda$-products given by \eqref{lambdap}. In this case, the extending datum $\Omega(R,Q)=(\varphi,\psi,l,r,g_\lambda(\cdot,\cdot), \circ_\lambda)$
is called a {\bf left-symmetric conformal extending structure } of $R$ by $Q$. Then we denote by $\mathfrak{L}(R,Q)$ the set of all left-symmetric conformal extending structures of $R$ by $Q$.
\end{definition}

By \eqref{lambdap}, the following relations hold in $R\natural Q$ for all $a,b\in R$ and $x,y\in Q$:
\begin{equation}
\begin{aligned}
&a_\lambda y=\psi(y)_{-\lambda-\partial}a+l(a)_\lambda y,\;\;x_\lambda b=\varphi(x)_\lambda b+r(b)_{-\lambda-\partial}x,\\
&x_\lambda y=g_\lambda(x,y)+ x\circ_\lambda y.\label{fomula}
\end{aligned}
\end{equation}

Next, we present a necessary and sufficient condition for $R\natural Q $ to be a left-symmetric conformal algebra with the
$\lambda$-products defined by \eqref{lambdap}.

\begin{theorem}\label{lth}
Let $R$ be a left-symmetric conformal algebra, $Q$ be a $\mathbb{C}[\partial]$-module and $\Omega(R,Q)$ an extending datum of $R$ by $Q$. Then the following
statements are equivalent:\\
(i) $R\natural Q$ is a left-symmetric conformal algebra with the $\lambda$-products given by \eqref{lambdap}.\\
(ii) The following compatibility conditions hold for all $a,b\in R$ and $x,y,z\in Q$:
\begin{flalign}
&(\varphi(x)_\lambda a-\psi(x)_{\lambda}a)_{\lambda+\mu}b+\varphi(r(a)_{\mu}x-l(a)_\mu x)_{\lambda+\mu}b
=\varphi(x)_\lambda(a_\mu b)&\tag{LC1}\label{LC1}\\
&-a_\mu(\varphi(x)_\lambda b)-\psi(r(b)_{-\lambda-\partial}x)_{-\mu-\partial}a,\nonumber\\
&r(b)_{-\lambda-\mu-\partial}(r(a)_{\mu}x-l(a)_\mu x)=r(a_\mu b)_{-\lambda-\partial}x-l(a)_\mu(r(b)_{-\lambda-\partial} x),\tag{LC2}\label{LC2}\\
&\psi(x)_{-\lambda-\mu-\partial}(a_\lambda b-b_\mu a)=a_\lambda(\psi(x)_{-\mu-\partial}b)-b_\mu(\psi(x)_{-\lambda-\partial}a)
+\psi(l(b)_\mu x)_{-\lambda-\partial} a\tag{LC3}\label{LC3}\\
&-\psi(l(a)_\lambda x)_{-\mu-\partial}b,\nonumber\\
&l(a_\lambda b)_{\lambda+\mu}x-l(b_\mu a)_{\lambda+\mu}x=l(a)_\lambda(l(b)_\mu x)-l(b)_\mu(l(a)_\lambda x),\tag{LC4}\label{LC4}\\
&\psi(y)_{-\lambda-\mu-\partial}(\psi(x)_{\mu}a-\varphi(x)_\mu a )+g_{\lambda+\mu}(l(a)_\lambda x,y)-a_\lambda(g_\mu(x,y))
-\psi(x\circ_\mu y)_{-\lambda-\partial}a\tag{LC5}\label{LC5}\\
&=g_{\lambda+\mu}(r(a)_{-\mu-\partial}x,y)-\varphi(x)_\mu(\psi(y)_{-\lambda-\partial}a)-g_\mu (x,l(a)_\lambda y),\nonumber\\
&(l(a)_\lambda x)\circ_{\lambda+\mu}y+l(\psi(x)_{-\lambda-\partial} a)_{\lambda+\mu}y-l(a)_\lambda(x\circ_\mu y)=l(\varphi(x)_\mu a)_{\lambda+\mu}y\tag{LC6}\label{LC6}\\
&+(r(a)_{-\mu-\partial}x)\circ_{\lambda+\mu}y-r(\psi(y)_{-\lambda-\partial}a)_{-\mu-\partial} x-x\circ_\mu(l(a)_\lambda y),\nonumber\\
&(g_\lambda(x,y)-g_\mu(y,x))_{\lambda+\mu}a+\varphi(x\circ_\lambda y-y\circ_\mu x)_{\lambda+\mu}a=\varphi(x)_\lambda(\varphi(y)_\mu a)
-\varphi(y)_\mu(\varphi(x)_\lambda a)\tag{LC7}\label{LC7}\\
&+g_\lambda(x,r(a)_{-\mu-\partial}y)-g_\mu(y,r(a)_{-\lambda-\partial}x),\nonumber\\
&r(a)_{-\lambda-\mu-\partial}(x\circ_\lambda y-y\circ_\mu x)=r(\varphi(y)_\mu a)_{-\lambda-\partial}x-r(\varphi(x)_\lambda a)_{-\mu-\partial}y\tag{LC8}\label{LC8}\\
&+x\circ_\lambda(r(a)_{-\mu-\partial}y)-y\circ_\mu(r(a)_{-\lambda-\partial}x),\nonumber\\
&\psi(z)_{-\lambda-\mu-\partial}g_\lambda(x,y)+g_{\lambda+\mu}(x\circ_\lambda y,z)-\varphi(x)_\lambda g_\mu(y,z)-g_\lambda(x,y\circ_\mu z)\tag{LC9}\label{LC9}\\
&=\psi(z)_{-\lambda-\mu-\partial}g_\mu(y,x)+g_{\lambda+\mu}(y\circ_\mu x,z)-\varphi(y)_\mu g_\lambda(x,z)-g_\mu(y,x\circ_\lambda z),\nonumber\\
&l(g_\lambda(x,y))_{\lambda+\mu}z+(x\circ_\lambda y)\circ_{\lambda+\mu}z-r(g_\mu(y,z))_{-\lambda-\partial}x-x\circ_\lambda(y\circ_\mu z)\tag{LC10}\label{LC10}\\
&=l(g_\mu(y,x))_{\lambda+\mu}z+(y\circ_\mu x)\circ_{\lambda+\mu}z-r(g_\lambda(x,z))_{-\mu-\partial}y-y\circ_\mu(x\circ_\lambda z).\nonumber
\end{flalign}
\end{theorem}
\begin{proof}
Since $\varphi,\psi,l,r$ are $\mathbb{C}[\partial]$-module homomorphisms and $g_\lambda(\cdot,\cdot), \circ_\lambda$ are conformal bilinear maps, conformal sesquilinearity for \eqref{lambdap} is
naturally satisfied.

Define
\begin{align*}
&F(a+x,b+y,c+z)=((a+x)_\lambda(b+y))_{\lambda+\mu}(c+z)-(a+x)_\lambda((b+y)_\mu(c+z))\\
& -((b+y)_\mu(a+x))_{\lambda+\mu}(c+z)+(b+y)_\mu((a+x)_\lambda(c+z)),\;\;\;\;a, b, c\in R,\;\; x, y, z\in Q.
\end{align*}
Note that $R\natural Q$ is a left-symmetric conformal algebra if and only if $F(a+x, b+y, c+z)=0$ for all $a, b, c\in R$ and $ x, y, z\in Q$.

Therefore, we only need to prove that $F(a+x, b+y, c+z)=0$ for all $a,b,c\in R$ and $x,y,z\in Q$ if and
only if \eqref{LC1}-\eqref{LC10} hold.
By the left-symmetry of left-symmetric conformal algebras,
we have that  $F(a+x, b+y, c+z)=0$ holds for all $a,b,c\in R$ and $x,y,z\in Q$ if and only if
\begin{align*}
&F(a,b,c)=0, ~~ F(x,b,c)=0,~~F(a,b,z)=0, \\
&F(a,y,z)=0,~~F(x,y,c)=0,~~F(x,y,z)=0,
\end{align*}
are satisfied for all $a,b,c\in R$ and $x,y,z\in Q$.

Notice that $F(a,b,c)=0$ for all $a$, $b$, $c\in R$ if and only if $R$ is a left-symmetric conformal algebra.
Since
\begin{align*}
&F(a,x,y)\\
=&(\psi(x)_{-\lambda-\partial}a+l(a)_\lambda x)_{\lambda+\mu}y-a_\lambda(g_\mu(x,y)+x\circ_\mu y)\\
&-(\varphi(x)_\mu a+r(a)_{-\mu-\partial} x)_{\lambda+\mu}y+x_\mu(\psi(y)_{-\lambda-\partial}a+l(a)_\lambda y)\\
=&(\psi(y)_{-\lambda-\mu-\partial}(\psi(x)_{-\lambda-\partial}a)+g_{\lambda+\mu}(l(a)_\lambda x,y))+((l(a)_\lambda x)\circ_{\lambda+\mu}y+l(\psi(x)_{-\lambda-\partial} a)_{\lambda+\mu}y)\\
&-((a_\lambda(g_\mu(x,y))+\psi(x\circ_\mu y)_{-\lambda-\partial}a)+l(a)_\lambda(x\circ_\mu y))\\
&-((\psi(y)_{-\lambda-\mu-\partial}(\varphi(x)_\mu a )+g_{\lambda+\mu}(r(a)_{-\mu-\partial}x,y))+(l(\varphi(x)_\mu a)_{\lambda+\mu}y+(r(a)_{-\mu-\partial}x)\circ_{\lambda+\mu}y ))\\
&+(\varphi(x)_\mu(\psi(y)_{-\lambda-\partial}a)+g_\mu (x,l(a)_\lambda y))+(r(\psi(y)_{-\lambda-\partial}a)_{-\mu-\partial} x+x\circ_\mu(l(a)_\lambda y))\\
=&0,
\end{align*}
$F(a,x,y)=0 $ if and only if \eqref{LC5} and \eqref{LC6} hold.
Similarly, we can get the following results: $F(x,a,b)=0$ if and only if \eqref{LC1} and \eqref{LC2} hold; $F(a,b,x)=0 $ if and only if \eqref{LC3} and \eqref{LC4} hold; $ F(x,y,a)=0$ if and only if \eqref{LC7} and \eqref{LC8} hold; $F(x,y,z)=0 $ if and only if
 \eqref{LC9} and \eqref{LC10} hold. Then the proof is completed.

\end{proof}
\begin{remark}
In fact, \eqref{LC2} and \eqref{LC4} mean that $(Q,l,r) $ is a bimodule of $R$. In addition, \eqref{LC1},
\eqref{LC3},  \eqref{LC6} and \eqref{LC8} are the compatibility conditions defining a matched pair of
left-symmetric conformal algebras (see \cite[Theorem 3.14]{HL2}).
\end{remark}
\begin{corollary}\label{coro3.4}
Let $\Omega(R,Q)=(\varphi,\psi,l,r,g_\lambda(\cdot,\cdot), \circ_\lambda)$ be a left-symmetric conformal
extending structure of $R$ by $Q$. Define conformal bilinear maps $ \triangleleft_\lambda$ :
$Q\times \mathfrak{g}(R)\rightarrow Q[\lambda] $
by $x\triangleleft_\lambda a=r(a)_{-\lambda-\partial} x-l(a)_{-\lambda-\partial} x $, $\triangleright_\lambda $:
$Q\times \mathfrak{g}(R)\rightarrow \mathfrak{g}(R)[\lambda] $ by
$x\triangleright_\lambda a=\varphi(a)_\lambda x-\psi(a)_\lambda x $, $f_\lambda$ : $Q\times Q\rightarrow \mathfrak{g}(R)[\lambda] $ by
$f_\lambda(x,y)=g_\lambda(x,y)-g_{-\lambda-\partial}(y,x) $
and $\{\cdot_\lambda\cdot\}$ : $Q\times Q\rightarrow Q[\lambda]$ by $\{x_\lambda y\}=x\circ_\lambda y-y\circ_{-\lambda-\partial}x$ for all $a\in R$
and $x,y\in Q$. Then $\Omega(\mathfrak{g}(R),Q)=(\triangleleft_\lambda,\triangleright_\lambda,f_\lambda,\{\cdot_\lambda\cdot\})$ is a Lie
conformal extending structure of $\mathfrak{g}(R) $ by $Q$ (see \cite[Definition 3.1]{HS}).

\end{corollary}
\begin{proof}
Let $R\natural Q$ be the unified product of $R$ and $\Omega(R,Q)$. Then the $\lambda$-brackets on $\mathfrak{g}(R\natural Q)$ are given by
\begin{eqnarray*}
&&[(a+x)_\lambda(b+y)]\\
&=& (a+x)_\lambda (b+y)-(b+y)_{-\lambda-\partial}(a+x)\\
&=&(a_\lambda b-b_{-\lambda-\partial}a+(\varphi(x)_\lambda b-\psi(x)_\lambda b)-(\varphi(y)_{-\lambda-\partial} a-\psi(y)_{-\lambda-\partial} a)+g_\lambda(x,y)-g_{-\lambda-\partial}(y,x))\\
&& +(x\circ_\lambda y-y\circ_{-\lambda-\partial}x+(l(a)_\lambda y-r(a)_\lambda y)-(l(b)_{-\lambda-\partial}x-r(b)_{-\lambda-\partial} x))\\
&=&([a_\lambda b]+x\triangleright_\lambda b-y\triangleright_{-\lambda-\partial}a+f_\lambda(x,y))+(\{x_\lambda y\}+x\triangleleft_\lambda b
-y\triangleleft_{-\lambda-\partial}a),\;\;\; a, b\in R, x, y\in Q.
\end{eqnarray*}
Therefore, $\Omega(\mathfrak{g}(R),Q)=(\triangleleft_\lambda,\triangleright_\lambda,f_\lambda,\{\cdot_\lambda\cdot\})$ is a Lie
conformal extending structure of $\mathfrak{g}(R) $ by $Q$.
\end{proof}

We present an example of left-symmetric conformal extending structures. More examples will be given in Section \ref{sec5}.
\begin{example}\cite[Proposition 3.7]{HL2}
Let $\Omega(R,Q)=(\varphi,\psi,l,r,g_\lambda(\cdot,\cdot),\circ_\lambda)$ be an extending datum of a left-symmetric conformal algebra $R$ by a $\mathbb{C}[\partial]$-module $Q$ where $\varphi,\psi,g_\lambda(\cdot,\cdot)$ and $\circ_\lambda$ are trivial. Denote this extending datum simply by $\Omega(R,Q)=(l,r) $. Then $\Omega(R,Q)$ is a left-symmetric conformal extending structure of $R$ by $Q$ if and only if $(Q,l,r)$ is an $R$-bimodule i.e. \eqref{LC2} and \eqref{LC4}
are satisfied. The associated unified product $R\natural Q$ denoted by $R\natural_{l,r} Q$ is called the {\bf semi-direct product} of $R$ and $Q$. The $\lambda$-products on $R\natural_{l,r} Q$ are given by
$$
(a+x)_\lambda(b+y)=a_\lambda b+l(a)_\lambda x+r(b)_{-\lambda-\partial}y,
$$
for all $a,b\in R$ and $x,y\in Q$.
\end{example}

It is straightforward to see that $R$ is a left-symmetric conformal subalgebra of $R\natural Q$. Therefore, for any $\Omega(R,Q)$, the unified product of $R$ and $\Omega(R,Q)$ satisfies the condition in the $\mathbb{C}[\partial]$-split extending structures problem. Next, we show that any left-symmetric conformal algebra $E=R\oplus Q$ containing $R$ as a subalgebra is isomorphic to a unified product.

\begin{theorem}\label{th35}
Let $R$ be a left-symmetric conformal algebra and $Q$ a $\mathbb{C}[\partial]$-module. Set $E=R\oplus Q$ where the direct sum is the sum of $\mathbb{C}[\partial]$-modules. Suppose that $(E, \cdot_\lambda \cdot)$ is a left-symmetric  conformal algebra containing
$R$ as a subalgebra. Then there exists a left-symmetric conformal extending structure
$\Omega(R,Q)= (\varphi,\psi,l,r,g_\lambda(\cdot,\cdot), \circ_\lambda)$ of $R$ by $Q$ such that  $E\cong R\natural Q $ as left-symmetric conformal algebras which stabilizes $R$ and co-stabilizes $Q$, where $R\natural Q$ is the unified product of $R$ and $\Omega(R,Q)$.
\end{theorem}
\begin{proof}
Since $E=R\oplus Q$, there exists a canonical $\mathbb{C}[\partial]$-module homomorphism $p$: $E\rightarrow R$ such that $p(a)=a$ for all $a\in R$. Thus we define an extending datum
$\Omega(R,Q)= (\varphi,\psi,l,r,g_\lambda(\cdot,\cdot), \circ_\lambda)$ of $R$ by $Q$ as follows:
$$
\begin{aligned}
&\varphi_:Q\rightarrow Cend(R),~~\varphi(x)_\lambda a:=p(x_\lambda a),\\
&\psi:Q\rightarrow Cend(R),~~\psi(x)_\lambda a:=p(a_{-\lambda-\partial}x),\\
&l:R\rightarrow Cend(Q),~~l(a)_\lambda x:=a_\lambda x-p(a_\lambda x),\\
&r:R\rightarrow Cend(Q),~~r(a)_\lambda x:=x_{-\lambda-\partial}a-p(x_{-\lambda-\partial} a),\\
&g_\lambda(\cdot,\cdot):Q\times Q\rightarrow R[\lambda],~~g_\lambda(x,y):=p(x_\lambda y),\\
&\circ_\lambda:Q\times Q\rightarrow Q[\lambda],~~x\circ_\lambda y:=x_\lambda y-p(x_\lambda y),
\end{aligned}
$$
for all $a\in R$ and $x,y\in Q$. By a similar proof as that in \cite[Theorem 2.4]{HS}, one can show that $\Omega(R,Q)= (\varphi,\psi,l,r,g_\lambda(\cdot,\cdot), \circ_\lambda)$ is a left-symmetric conformal
structure and $E\cong R\natural Q$ as left-symmetric conformal algebras, which stabilizes $R$ and co-stabilizes $Q$.

\delete{
 We can define a map $ \tau$ : $R\times Q\rightarrow E $ by $\tau(a,x)=a+x $. It is clear that $ \tau$ is a
$\mathbb{C}[\partial]$-module isomorphism, where the inverse of $\tau $ is $\tau^{-1}$ : $ E\rightarrow R\times Q$ which is defined by
$\tau^{-1}(e)=(p(e),e-p(e)) $ where $e$ is an identity element in $E$. If $\tau $: $ R\times Q\rightarrow E$ is also an isomorphism of
left-symmetric (resp. Novikov) conformal algebras, there exists a unique $\lambda$-product on $R\times Q$ given by
\begin{equation}
(a,x)_\lambda(b,y)=\tau^{-1}(\tau(a,x)_\lambda\tau(b,y))\label{lambda3}
\end{equation}
for all $a,b\in R$ and $x,y\in Q$. Thereby , we only need to check that the $\lambda$-product defined by \eqref{lambda3} is just one given by
\eqref{lambdap} related with the above extending system
$\Omega(R,Q)= (\varphi,\psi,l,r,g_\lambda,\cdot\circ_\lambda\cdot)$ to complete this proof. It is check as follows:
for any $a,b\in R$ and $x,y\in Q$
$$
\begin{aligned}
&(a,x)_\lambda(b,y)\\
=&\tau^{-1}(\tau(a,x)_\lambda\tau(b,y))\\
=&\tau^{-1}((a+x)_\lambda(b+y))\\
=&(p(a_\lambda b+x_\lambda b+a_\lambda y+x_\lambda y),a_\lambda+x_\lambda b+a_\lambda y+x_\lambda y-p(a_\lambda b+x_\lambda b+a_\lambda y+x_\lambda y))\\
=&(a_\lambda b+\varphi(x)_\lambda b+\psi(y)_{-\lambda-\partial}a+g_\lambda(x,y),r(b)_{-\lambda-\partial}x+l(a)_\lambda y+<x_\lambda y> ).
\end{aligned}
$$
Therefore, it is clear that to see that $\tau$ stabilizes $R$ and co-stabilizes $Q$. Then, the proof is finished.}
\end{proof}

\begin{definition}\label{3.10}
Let $R$ be a left-symmetric conformal algebra and $Q$ a $\mathbb{C}[\partial]$-module. If there exists a pair of $\mathbb{C}[\partial]$-module
homomorphisms $(u,v)$, where $u$: $Q\rightarrow R$, $v\in Aut_{\mathbb{C}[\partial]}(Q)$ such that the left-symmetric  conformal extending structure
$\Omega(R,Q)= (\varphi,\psi,l,r,g_\lambda(\cdot,\cdot), \circ_\lambda )$ can be obtained from another corresponding
extending structure $\Omega'(R,Q)= (\varphi',\psi',l',r',g'_\lambda(\cdot,\cdot),\circ'_\lambda) $ using $(u,v)$
as follows:
\begin{align}
&\psi(x)_{-\lambda-\partial}a+u(l(a)_\lambda x)=a_\lambda u(x)+\psi'(v(x))_{-\lambda-\partial}a,\label{eq1}\\
&v(l(a)_\lambda x)=l'(a)_\lambda v(x),\label{eq2}\\
&\varphi(x)_\lambda a+u(r(a)_{-\lambda-\partial}x)=u(x)_\lambda a+\varphi'(v(x))_\lambda a,\label{eq3}\\
&v(r(a)_{-\lambda-\partial}x)=r'(a)_{-\lambda-\partial}v(x),\label{eq4}\\
&g_\lambda(x,y)+u(x\circ_\lambda y)=u(x)_\lambda u(y)+\varphi'(v(x))_\lambda u(y)+\psi'(v(y))_{-\lambda-\partial}u(x)+g'_\lambda(v(x),v(y)),\label{eq5}\\
&v(x\circ_\lambda y)=r'(u(y))_{-\lambda-\partial}v(x)+l'(u(x))_\lambda v(y)+v(x)\circ'_\lambda v(y),\label{eq6}
\end{align}
for all $a,b\in R$ and $x,y\in Q$, then $\Omega(R,Q)$ and $\Omega'(R,Q)$ are called {\bf equivalent} and we denote it by $\Omega(R,Q)\equiv\Omega'(R,Q)$.
In particular, if $v=\Id$, $\Omega(R,Q)$ and $\Omega'(R,Q)$ are called {\bf cohomologous} and we denote it by
$\Omega(R,Q)\approx\Omega'(R,Q) $.
\end{definition}
\begin{lemma}\label{lemma31}
Let $\Omega(R,Q)= (\varphi,\psi,l,r,g_\lambda(\cdot,\cdot), \circ_\lambda)$ and
$\Omega'(R,Q)= (\varphi',\psi',l',r',g'_\lambda(\cdot,\cdot), \circ'_\lambda)$ be two left-symmetric
conformal extending structures of $R$ by $Q$ and $R\natural Q$, $R\natural' Q$ be the corresponding unified products. Thus $R\natural Q\equiv R\natural' Q$
if and only if $\Omega(R,Q)\equiv\Omega'(R,Q)$. Moreover, $R\natural Q\approx R\natural' Q $ if and only if $\Omega(R,Q)\approx\Omega'(R,Q) $.
\end{lemma}
\begin{proof}
Let $\tau$ : $R\natural Q\rightarrow R\natural' Q $ be a homomorphism of left-symmetric
conformal algebras which stabilizes $R$. Since $\tau$ stabilizes $R$, then $\tau(a)=a$ for all $a\in R$. Thus, we set $\tau(a+x)=(a+u(x))+v(x)$ for all $a\in R$ and $x\in Q$,
where $u$: $Q\rightarrow R$, $v$: $Q\rightarrow Q$ are two linear maps. Similar to the proof in \cite[Lemma 3.6]{HS}, it is straightforward to check that  $\tau$ is a left-symmetric
conformal algebra isomorphism
if and only if $u$ is a $\mathbb{C}[\partial]$-module homomorphism, $v\in Aut_{\mathbb{C}[\partial]}(Q)$ and \eqref{eq1}-\eqref{eq6} hold. Moreover,  it is easy to see that a left-symmetric conformal algebra isomorphism $\tau$ co-stabilizes $Q$
if and only if $u$ is a $\mathbb{C}[\partial]$-module homomorphism, $v=\Id_Q$ and \eqref{eq1}-\eqref{eq6} hold.
\end{proof}
\begin{theorem} \label{th3.12}
Let $R$ be a left-symmetric conformal algebra, $Q$ be a $\mathbb{C}[\partial]$-module, and $E=R\oplus Q$ where the direct sum is the sum of $\mathbb{C}[\partial]$-modules. Then we have\\
(i) \delete{`` $\equiv$ " is an equivalence relation on the set $\mathfrak{L}(R,Q)$ of all left-symmetric conformal extending structures of $R$ by $Q$.}
 Set $\mathcal{H}^2_R(Q,R):= \mathfrak{L}(R,Q)/\equiv$. Then the map
 \begin{equation}
 \mathcal{H}^2_R(Q,R)\rightarrow CExtd(E,R), ~\overline{\Omega(R,Q)}\mapsto (R\natural Q,\cdot_\lambda\cdot)
 \end{equation}
 is bijective, where $\overline{\Omega(R,Q)} $ is the equivalence class of $ \Omega(R,Q)$ under $\equiv$.\\
(ii) Set $\mathcal{H}^2(Q,R):= \mathfrak{L}(R,Q)/\approx$. Then the map
 \begin{equation}
 \mathcal{H}^2(Q,R)\rightarrow CExtd'(E,R), ~\overline{\overline{\Omega(R,Q)}}\mapsto (R\natural Q,\cdot_\lambda\cdot)
 \end{equation}
is bijective, where $\overline{\overline{\Omega(R,Q)}} $ is the equivalence class of $ \Omega(R,Q)$ under $\approx$.
\end{theorem}
\begin{proof}
It follows from Theorem \ref{lth}, Theorem \ref{th35} and Lemma \ref{lemma31}.
\end{proof}
\begin{remark}
 By Theorem \ref{th3.12}, $\mathcal{H}^2_R(Q,R)$ classifies all left-symmetric conformal algebra structures on $E=R\oplus Q$ containing $R$ as a subalgebra up to isomorphism that stabilizes $R$. Thus, $\mathcal{H}^2_R(Q,R)$ provides a theoretical answer to the $\mathbb{C}[\partial]$-split extending structures problem.
 Finally, all left-symmetric conformal algebra structures on $E=R\oplus Q$ containing $R$ as a subalgebra up to isomorphism that stabilizes $R$ and co-stabilizes $Q$  are characterized by $\mathcal{H}^2(Q,R)$.
\end{remark}

\section{Unified products when $Q=\mathbb{C}[\partial]x$}
In this section, we investigate the general unified products when $R$ is a free $\mathbb{C}[\partial]$-module and $Q$ is a free $\mathbb{C}[\partial]$-module of rank 1. Set $Q =\mathbb{C}[\partial]x$.

\delete{we set $R$ be a left-symmetric conformal algebra, $Q$ be a  $ \mathbb{C}[\partial]$-module and $E$ be that in the $\mathbb{C}[\partial]$-split extending structures problem. If $Q$ is a torsion $\mathbb{C}[\partial]$-module, on account of the fact that the torsion element must be a central element, $Q$ is a two-sided ideal contained in the center of $E$. Hence in this case, $E\cong R\oplus Q$ as left-symmetric conformal algebras, and $Q$ is a trivial  two-sided ideal of $E$.}
\begin{definition}
Let $R=\mathbb{C}[\partial]V$ be a left-symmetric conformal algebra which is a free $\mathbb{C}[\partial]$-module. A {\bf flag datum} of $R$ is a
sextuple $(h_\lambda(\cdot,\partial),k_\lambda(\cdot,\partial), D_\lambda,T_\lambda,M(\lambda,\partial),P(\lambda,\partial) )$, where
$P(\lambda,\partial)$ $\in\mathbb{C}[\lambda,\partial] $, $M(\lambda,\partial)\in R[\lambda] $,
$h_\lambda(\cdot,\partial):R\rightarrow \mathbb{C}[\lambda,\partial] $ and $k_\lambda(\cdot,\partial):R\rightarrow \mathbb{C}[\lambda,\partial] $ are two left conformal linear maps,
$D_\lambda:R\rightarrow R[\lambda] $ and $T_\lambda:R\rightarrow R[\lambda]  $
are conformal linear maps satisfying the following conditions for all $a,b\in V$ :
\begin{align}
&(D_\lambda(a)-T_\lambda(a))_{\lambda+\mu}b
+(k_\mu(a,-\lambda-\mu)-h_\mu(a,-\lambda-\mu) )D_{\lambda+\mu}(b)\label{lfd1}\\
&=D_\lambda(a_\mu b)-a_\mu(D_\lambda(b))-k_{-\lambda-\mu-\partial}(b,\mu+\partial)T_{-\mu-\partial}(a),\nonumber\\
&(k_\mu(a,-\lambda-\mu)-h_\mu(a,-\lambda-\mu))k_{-\lambda-\mu-\partial}(b,\partial)
=k_{-\lambda-\partial}(a_\mu b,\partial)-k_{-\lambda-\mu-\partial}(b,\mu+\partial)h_\mu(a,\partial),\label{lfd2}\\
&T_{-\lambda-\mu-\partial}(a_\lambda b-b_\mu a)=a_\lambda T_{-\mu-\partial}(b)-b_\mu T_{-\lambda-\partial}(a)
+h_\mu(b,\lambda+\partial)T_{-\lambda-\partial}(a)\label{lfd3}\\
&-h_\lambda(a,\mu+\partial)T_{-\mu-\partial}(b),\nonumber\\
&h_{\lambda+\mu}(a_\lambda b-b_\mu a,\partial)=h_\mu(b,\lambda+\partial)h_\lambda(a,\partial)-h_\lambda(a,\mu+\partial)h_\mu(b,\partial),\label{lfd4}\\
&T_{-\lambda-\mu-\partial}(T_{\mu}(a)-D_\mu(a))+h_\lambda(a,-\lambda-\mu)M(\lambda+\mu,\partial)
-P(\mu,\lambda+\partial)T_{-\lambda-\partial}(a)\label{lfd5}\\
&-a_\lambda M(\mu,\partial)=k_\lambda(a,-\lambda-\mu)M(\lambda+\mu,\partial)-D_\mu(T_{-\lambda-\partial}(a))
-h_\lambda(a,\mu+\partial)M(\mu,\partial),\nonumber\\
&h_\lambda(a,-\lambda-\mu)P(\lambda+\mu,\partial)+h_{\lambda+\mu}(T_{-\lambda-\partial}(a),\partial)-P(\mu,\lambda+\partial)h_\lambda(a,\partial)
=h_{\lambda+\mu}(D_\mu(a),\partial)\label{lfd6}\\
&+k_\lambda(a,-\lambda-\mu)P(\lambda+\mu,\partial)-k_{-\mu-\partial}(T_{-\lambda-\partial}(a),\partial)
-h_\lambda(a,\mu+\partial)P(\mu,\partial),\nonumber\\
&(M(\lambda,\partial)-M(\mu,\partial))_{\lambda+\mu}a+(P(\lambda,-\lambda-\mu)-P(\mu,-\lambda-\mu))D_{\lambda+\mu}(a)=D_\lambda(D_\mu(a))\label{lfd7}\\
&-D_\mu(D_\lambda(a))+k_{-\lambda-\mu-\partial}(a,\lambda+\partial)M(\lambda,\partial)
-k_{-\lambda-\mu-\partial}(a,\mu+\partial)M(\mu,\partial),\nonumber\\
&(P(\lambda,-\lambda-\mu)-P(\mu,-\lambda-\mu))k_{-\lambda-\mu-\partial}(a,\partial)
=k_{-\lambda-\partial}(D_\mu(a),\partial)-k_{-\mu-\partial}(D_\lambda(a),\partial)\label{lfd8}\\
&+k_{-\lambda-\mu-\partial}(a,\lambda+\partial)P(\lambda,\partial)
-k_{-\lambda-\mu-\partial}(a,\mu+\partial)P(\mu,\partial),\nonumber\\
&T_{-\lambda-\mu-\partial}(M(\lambda,\partial)-M(\mu,\partial))+(P(\lambda,-\lambda-\mu)
-P(\mu,-\lambda-\mu))M(\lambda+\mu,\partial)\label{lfd9}\\
&=D_\lambda(M(\mu,\partial))- D_\mu(M(\lambda,\partial))+P(\mu,\lambda+\partial)M(\lambda,\partial)
-P(\lambda,\mu+\partial)M(\mu,\partial),\nonumber\\
&h_{\lambda+\mu}(M(\lambda,\partial)-M(\mu,\partial),\partial)+(P(\lambda,-\lambda-\mu)-P(\mu,-\lambda-\mu)) P(\lambda+\mu,\partial)\label{lfd10}\\
&=k_{-\lambda-\partial}(M(\mu,\partial),\partial)-k_{-\mu-\partial}(M(\lambda,\partial),\partial)+P(\mu,\lambda+\partial)P(\lambda,\partial)
-P(\lambda,\mu+\partial)P(\mu,\partial).\nonumber
\end{align}
We denote the set of all flag datums of the left-symmetric conformal algebra $R$ by $\mathcal{FLC}(R)$.
\end{definition}
\begin{proposition}\label{pro4.2}
Let $R=\mathbb{C}[\partial]V$ be a left-symmetric conformal algebra which is a free $\mathbb{C}[\partial]$-module and $Q=\mathbb{C}[\partial]x$ a free $\mathbb{C}[\partial]$-module of rank 1. Then there is a bijection between the set $\mathfrak{L}(R,Q)$
 of all left-symmetric  conformal extending structures of $R$ by $Q$ and $\mathcal{FLC}(R)$.
\end{proposition}
\begin{proof}
Given a left-symmetric conformal extending structure
$\Omega(R,Q)$
= $(\varphi,\psi,l,r,g_\lambda(\cdot,\cdot), \circ_\lambda )$. Since
$Q=\mathbb{C}[\partial]x$ is a free $\mathbb{C}[\partial]$-module of rank 1, we  set
$$
\begin{aligned}
&l(a)_\lambda x=h_\lambda(a,\partial)x,~~r(a)_\lambda x=k_\lambda(a,\partial)x,~~\varphi(x)_\lambda a=D_\lambda(a),\\
&\psi(x)_\lambda a=T_\lambda(a),~~x\circ_\lambda x=P(\lambda,\partial)x,~~g_\lambda(x,x)=M(\lambda,\partial),
\end{aligned}
$$
where $P(\lambda,\partial)\in\mathbb{C}[\lambda,\partial] $, $M(\lambda,\partial)\in R[\lambda] $,
$h_\lambda(\cdot,\partial):R\rightarrow \mathbb{C}[\lambda,\partial] $ and $k_\lambda(\cdot,\partial):R\rightarrow \mathbb{C}[\lambda,\partial] $ are left conformal linear maps,
$D_\lambda:R\rightarrow R[\lambda] $ and $T_\lambda:R\rightarrow R[\lambda]  $ are conformal linear maps.
It is straightforward to check that the conditions \eqref{LC1}-\eqref{LC10} in Theorem \ref{lth} are equivalent to \eqref{lfd1}-\eqref{lfd10}.

\end{proof}

The left-symmetric conformal algebra corresponding to the flag datum
$(h_\lambda(\cdot,\partial),k_\lambda(\cdot,\partial), D_\lambda,$ $T_\lambda,M(\lambda,\partial),P(\lambda,\partial) )$ of $R$ is the
$\mathbb{C}[\partial]$-module $R\oplus\mathbb{C}[\partial]x$ with the following $\lambda$-products:
\begin{eqnarray}
&&(a+x)_\lambda (b+x)=(a_\lambda b+T_{-\lambda-\partial}(a)+D_\lambda(b)+M(\lambda,\partial))+(h_\lambda(a,\partial)+k_{-\lambda-\partial}(b,\partial)+P(\lambda,\partial))x,
\end{eqnarray}
 for all $a,b\in V $. Denote  this left-symmetric conformal algebra by $LC(R,\mathbb{C}[\partial]x|h_\lambda(\cdot,\partial),k_\lambda(\cdot,\partial),$\\
 $D_\lambda,T_\lambda,M(\lambda,\partial),P(\lambda,\partial))$.

\begin{theorem}\label{th4.3}
Let $R=\mathbb{C}[\partial]V$ be a left-symmetric conformal algebra which is free as a $\mathbb{C}[\partial]$-module and $Q=\mathbb{C}[\partial]x$ be a free $\mathbb{C}[\partial]$-module
of rank 1. Set $E=R\oplus Q$ where the direct sum is the sum of $\mathbb{C}[\partial]$-modules. Then we obtain\\
(1) $CExtd(E,R)\cong \mathcal{H}^2_R(Q,R)\cong \mathcal{FLC}(R)/\equiv$,
 where `` $\equiv$ " is the equivalence relation on the set
$\mathcal{FLC}(R) $ as follows:
$$
(h_\lambda(\cdot,\partial),k_\lambda(\cdot,\partial), D_\lambda,T_\lambda,M(\lambda,\partial),P(\lambda,\partial) )\equiv
(h'_\lambda(\cdot,\partial),k'_\lambda(\cdot,\partial), D'_\lambda,T'_\lambda,M'(\lambda,\partial),P'(\lambda,\partial) )
$$
if and only if
$h_\lambda(\cdot,\partial)=h'_\lambda(\cdot,\partial) $ , $k_\lambda(\cdot,\partial)=k'_\lambda(\cdot,\partial) $ and there exist $\omega\in R$ and
$\beta \in \mathbb{C}\setminus\{0\}$ such that for all $a\in V$:
\begin{align}
&D_\lambda(a)=\beta D'_\lambda(a)+{\omega}_\lambda a-k_{-\lambda-\partial}(a,\partial)\omega,\label{eq4.1}\\
&T_{-\lambda-\partial}(a)=\beta T'_{-\lambda-\partial}(a)+a_\lambda \omega-h_\lambda(a,\partial)\omega,\label{eq4.2}\\
&M(\lambda,\partial)={\omega}_\lambda \omega+\beta^2M'(\lambda,\partial)+\beta T'_{-\lambda-\partial}(\omega)+\beta D'_\lambda(\omega)-P(\lambda,\partial)\omega,\label{eq4.3}\\
&P(\lambda,\partial)=k'_{-\lambda-\partial}(\omega,\partial)+h'_\lambda(\omega,\partial)+\beta P'(\lambda,\partial).\label{eq4.4}
\end{align}
The bijection between $\mathcal{FLC}(R)/\equiv$ and $CExtd(E,R)$ is given by
\begin{equation}
\begin{aligned}
&\overline{(h_\lambda(\cdot,\partial),k_\lambda(\cdot,\partial), D_\lambda,T_\lambda,M(\lambda,\partial),P(\lambda,\partial) )}\rightarrow\\
&LC(R,\mathbb{C}[\partial]x|h_\lambda(\cdot,\partial),k_\lambda(\cdot,\partial), D_\lambda,T_\lambda,M(\lambda,\partial),P(\lambda,\partial)).
\end{aligned}
\end{equation}

\noindent(2) $CExtd'(E,R)\cong \mathcal{H}^2(Q,R)\cong \mathcal{FLC}(R)/\approx$, where `` $ \approx$ " is the equivalence
relation on the set
$\mathcal{FLC}(R) $ as follows:
$$(h_\lambda(\cdot,\partial),k_\lambda(\cdot,\partial), D_\lambda,T_\lambda,M(\lambda,\partial) ,P(\lambda,\partial))\approx
(h'_\lambda(\cdot,\partial),k'_\lambda(\cdot,\partial), D'_\lambda,T'_\lambda,M'(\lambda,\partial) ,P'(\lambda,\partial)) $$
if and only if
$h_\lambda(\cdot,\partial)=h'_\lambda(\cdot,\partial) $ , $k_\lambda(\cdot,\partial)=k'_\lambda(\cdot,\partial) $ and there exists $\omega\in R$
such that \eqref{eq4.1}-\eqref{eq4.4} hold for $\beta=1$. The bijection between
$\mathcal{FLC}(R)/\approx$ and $CExtd'(E,R)$ is given by
\begin{equation}
\begin{aligned}
&\overline{\overline{(h_\lambda(\cdot,\partial),k_\lambda(\cdot,\partial), D_\lambda,T_\lambda,M(\lambda,\partial),P(\lambda,\partial) )}}\rightarrow\\
&LC(R,\mathbb{C}[\partial]x|h_\lambda(\cdot,\partial),k_\lambda(\cdot,\partial), D_\lambda,T_\lambda,M(\lambda,\partial),P(\lambda,\partial)).
\end{aligned}
\end{equation}
\end{theorem}
\begin{proof}
Since $u$: $Q\rightarrow R $ is a $\mathbb{C}[\partial]$-module homomorphism and
$v$ is a $\mathbb{C}[\partial]$-module automorphism of $Q$  in Definition \ref{3.10}, we set $u(x)=\omega $ and $v(x)=\beta x $ where $\omega\in R$ and $\beta\in\mathbb{C}\setminus\{0\}$.
Therefore, we can directly get this theorem by Lemma \ref{lemma31}, Theorem \ref{th3.12} and Proposition \ref{pro4.2}.
\end{proof}

If $(h_\lambda(\cdot,\partial),k_\lambda(\cdot,\partial), D_\lambda,T_\lambda,M(\lambda,\partial),P(\lambda,\partial) )\in \mathcal{FLC}(R)$ with  $ h_\lambda(\cdot,\partial)$, $k_\lambda(\cdot,\partial)$ and $T_\lambda $ trivial, then we denote this flag datum by
$( D_\lambda,M(\lambda,\partial),P(\lambda,\partial)) $. The set of all such flag datums of $R$ is denoted by $\mathcal{DFLC}_1(R)$. In this case, notice that $D_\lambda$ is a conformal derivation of $R$. By Theorem \ref{th4.3},
$( D_\lambda,M(\lambda,\partial),P(\lambda,\partial))\equiv ( D'_\lambda,M'(\lambda,\partial),P'(\lambda,\partial))$ if and only if there exists a
pair $(\beta,\omega)\in \mathbb{C}\setminus\{0\}\times R$ such that for all $a\in V$,
\begin{align}
&D_\lambda(a)=\beta D'_\lambda(a)+{\omega}_\lambda a,\label{eq4.25}\\
&0=a_\lambda \omega,\label{eq4.26}\\
&M(\lambda,\partial)={\omega}_\lambda \omega+\beta^2M'(\lambda,\partial)+\beta D'_\lambda(\omega)-P(\lambda,\partial)\omega,\label{eq4.27}\\
&P(\lambda,\partial)=\beta P'(\lambda,\partial).\label{eq4.28}
\end{align}
Moreover, $( D_\lambda,M(\lambda,\partial),P(\lambda,\partial))\approx ( D'_\lambda,M'(\lambda,\partial),P'(\lambda,\partial))$ if and only if
\eqref{eq4.25}-\eqref{eq4.28} hold with $\beta=1$. Denote the set of all flag datums such as
$( D_\lambda,0,P(\lambda,\partial))$ by $\mathcal{DDFLC}_1(R)$.

If $(h_\lambda(\cdot,\partial),k_\lambda(\cdot,\partial), D_\lambda,T_\lambda,M(\lambda,\partial),P(\lambda,\partial) )\in \mathcal{FLC}(R)$ with $ P(\lambda,\partial)=0$,  $ h_\lambda(\cdot,\partial)$ and $T_\lambda $ trivial,  then we denote this flag datum by
$(k_\lambda(\cdot,\partial), D_\lambda,M(\lambda,\partial)) $. The set of all such flag datums of $R$ in which $ k_\lambda(\cdot,\lambda)\neq0$ is denoted by $\mathcal{DFLC}_2(R)$. In this case, notice that $D$ is a twisted conformal derivation of $R$. By Theorem \ref{th4.3},
$(k_\lambda(\cdot,\partial),D_\lambda,M(\lambda,\partial))\equiv( k'_\lambda(\cdot,\partial),D'_\lambda,M'(\lambda,\partial))$ if and only if there exists a
pair $(\beta,\omega)\in \mathbb{C}\setminus\{0\}\times R$ such that for all $a\in V$,
\begin{align}
&D_\lambda(a)=\beta D'_\lambda(a)+{\omega}_\lambda a-k_{-\lambda-\partial}(a,\partial)\omega,\label{eq4.29}\\
&0=a_\lambda \omega,\label{eq4.30}\\
&M(\lambda,\partial)={\omega}_\lambda \omega+\beta^2M'(\lambda,\partial)+\beta D'_\lambda(\omega),\label{eq4.31}\\\
&0=k'_{-\lambda-\partial}(\omega,\partial).\label{eq4.32}
\end{align}
In addition, $(k_\lambda(\cdot,\partial),D_\lambda,M(\lambda,\partial))\approx ( k'_\lambda(\cdot,\partial),D'_\lambda,M'(\lambda,\partial))$ if and only if
\eqref{eq4.29}-\eqref{eq4.32} hold with $\beta=1$.

\begin{corollary}
Let $R $ be a left-symmetric conformal algebra which is a free $\mathbb{C}[\partial]$-module.
If $h_\lambda(\cdot,\partial) $ is trivial in any flag datum of $R$ and all conformal semi-quasicentroids of $R$ are inner, then
\begin{equation}
CExtd(E,R)\cong H^2_R(Q,R)\cong(DFLC_1(R)/\equiv)\cup(DFLC_2(R)/\equiv)
\end{equation}
and
\begin{equation}
CExtd'(E,R)\cong H^2(Q,R)\cong(DFLC_1(R)/\approx)\cup(DFLC_2(R)/\approx).
\end{equation}
In addition, if there does not exist any non-zero element $b\in R$ such that $a_\lambda b=0$ for all $a\in R$, then
$$
CExtd(E,R)\cong H^2_R(Q,R)\cong(DDFLC_1(R)/\equiv_1)\cup(DFLC_2(R)/\equiv_2),
$$
where $(D_\lambda,0,P(\lambda,\partial))\equiv_1 (D'_\lambda,0,P'(\lambda,\partial)) $ if and only if there exists $\beta\in\mathbb{C}\setminus\{0\} $
such that
$$
D_\lambda (a)=\beta D'_\lambda(a),~~~~P(\lambda,\partial)=\beta P'(\lambda,\partial),
$$
and $(k_\lambda(\cdot,\partial),D_\lambda,M(\lambda,\partial))\equiv_2 (k'_\lambda(\cdot,\partial),D'_\lambda,M'(\lambda,\partial)) $ if and only if
$k_\lambda(\cdot,\partial)=k'_\lambda(\cdot,\partial) $ and there exists $\beta\in\mathbb{C}\setminus\{0\}$ such that
$$
D_\lambda(a)=\beta D'_\lambda(a),~~~~M(\lambda,\partial)=\beta^2M'(\lambda,\partial),
$$
and
$$
CExtd'(E,R)\cong H^2(Q,R)\cong DDFLC_1(R)\cup DFLC_2(R).
$$
\end{corollary}
\begin{proof}
By \eqref{lfd3}, $T_\lambda$ is a conformal semi-quasicentroid.  Since all conformal semi-quasicentroids of $R$ are inner, there exists some $b\in R$ such
that $T_{-\lambda-\partial}(a)=a_\lambda b$ for all $a\in R$. By Theorem \ref{th4.3}, we can make $T_{\lambda}=0$. Thus, one gets
$k_{\lambda}(a,-\lambda-\mu)M(\lambda+\mu,\partial)+a_\lambda M(\mu,\partial)=0$
and $k_{\lambda}(a,-\lambda-\mu)P(\lambda+\mu,\partial)=0$ from \eqref{lfd5} and \eqref{lfd6}. Therefore we obtain that there are two cases:
(1) $k_\lambda(\cdot,\partial)=0$; (2) $k_\lambda(\cdot,\partial)\neq0, P(\lambda,\partial)=0 $.
Then the first conclusion can be directly obtained by Theorem \ref{th4.3}.

Let us consider when $R$ also does not have any non-zero element $b$ such that $a_\lambda b=0$ for all $a\in R$. If $k_\lambda(\cdot,\partial)=0$, then
$a_\lambda M(\mu,\partial)=0 $ for all $a\in R$ by \eqref{lfd5}. Thus, one has $M(\mu,\partial)=0 $. Moreover, we also can obtain $\omega=0$ in \eqref{eq4.25}-\eqref{eq4.28} and \eqref{eq4.29}-\eqref{eq4.32} by \eqref{eq4.26} and \eqref{eq4.30}. Then we get the second conclusion by
Theorem \ref{th4.3}.
\end{proof}
\delete{\begin{remark}
It should be pointed out that this theorem is useful for computing some unified products of some special left-symmetric conformal algebras and $Q=\mathbb{C}[\partial]x$ in the next section.
\end{remark}}

\delete{In what follows, we prepare to use this theory characterize the extending structures of left-symmetric conformal algebra $R=\mathbb{C}[\partial]a$ defined by $a_\lambda a=0$ by $Q = \mathbb{C}[\partial]x$ with $M(\lambda,\partial)=0$. In the sequel, let us make some notations for convenience, and we will not be confused by these notations because $R$ is free and of rank 1 as a $\mathbb{C}[\partial]$-module. We denote
\begin{align*}
h_\lambda(a,\partial)=h(\lambda,\partial),~~~~~k_\lambda(a,\partial)=k(\lambda,\partial),~~~~~D_\lambda(a)=D(\lambda,\partial)a, ~~~~~ T_\lambda(a)=T(\lambda,\partial)a,
\end{align*}
where $ h(\lambda,\partial),k(\lambda,\partial),D(\lambda,\partial),T(\lambda,\partial), \in\mathbb{C}[\partial,\lambda]$.
\begin{proposition}
Let $R = \mathbb{C}[\partial]a$ be the left-symmetric conformal algebra with $a_\lambda a = 0$
and $Q = \mathbb{C}[\partial]x$. Then $\mathcal{H}^2_R(Q,R)$ can be described by the following kinds of
flag datums:\\
(a)~~$(0,0,D_\lambda,0,0)$, moreover, $(0,0,D_\lambda,0,0)$
is equivalent to $(0,0,D'_\lambda,0,0)$ if and only if there exists $\beta \in \mathbb{C }\backslash\{0\}$ such
that $D_\lambda = \beta D'_\lambda$;\\
(b)~~$(0,0,D_\lambda,1,1)$, where $T_\lambda(a)=a$ and $D_\lambda(a)=D(\lambda)a$, $D(\lambda)a\in\mathbb{C}[\lambda]$;\\
(c)~~$(0,0,D_\lambda,0,1)$ where $D_\lambda(a)=D(\lambda)a$ $D(\lambda)\in\mathbb{C}[\lambda]$;\\
(d)~~$(0,0,0,0,\lambda+\partial+c) $  where $\alpha,\beta,c\in \mathbb{C}$;\\
(e)~~$(0,0,D_\lambda,0,\lambda+\partial+c) $ where $D_\lambda(a)=(\partial+\alpha\lambda+\beta)a$, $\alpha,\beta,c\in \mathbb{C}$ ;\\
(f)~~$(0,0,D_\lambda,T_\lambda,\lambda+\partial+c) $ where $D_\lambda(a)=(\partial+\alpha\lambda+\beta)a$ $T_\lambda(a)=ca$, $\alpha,\beta,c\in \mathbb{C}$ ;\\
(g)~~$(0,0,D_\lambda,T_\lambda,\lambda+\partial+c) $  where $D_\lambda(a)=(\partial+\alpha\lambda+\beta)a$ $T_\lambda(a)=(-\lambda+c)a$, $\alpha,\beta,c\in \mathbb{C}$ ; \\
(h)~~$ (h_\lambda(\cdot,\partial),0,0,0,0)$ where $h_\lambda(x,\partial)=h(\lambda,\partial)x $, $h(\lambda,\partial)\in \mathbb{C }[\lambda]$;\\
(i)~~$ (h_\lambda(\cdot,\partial),0,0,T_\lambda,1)$ where $h_\lambda(x,\partial)=h(\lambda,\partial)x$, $h(\lambda,\partial)x\in \mathbb{C }$ and $T_\lambda a=a$;\\
(j)~~$ (h_\lambda(\cdot,\partial),0,0,T_\lambda,\lambda+\partial+c)$ where $h_\lambda(x,\partial)=h(\lambda)\in \mathbb{C }[\lambda]$ ;\\
\end{proposition}

\begin{proof}
By Theorem \ref{th4.3}, for characterizing the extending structures of $R$ by $Q$ up to equivalence, we only need to describe the set $\mathcal{FLC}(R)$ up to equivalence.\par
By assumption $a_\lambda a=0$, we get $h(\lambda,\mu+\partial)h(\mu,\partial)=h(\mu,\lambda+\partial)h(\lambda,\partial)$ by \eqref{lfd4}. Then we obtain
the degree of $\partial $ is smaller than 1 by comparing the degree of $\lambda$. Thus, taking $h(\lambda,\partial)=h(\lambda)$ into \eqref{lfd2}, we get
\begin{equation}
(k(\mu,-\lambda-\mu)-h(\mu))k(-\lambda-\mu-\partial,\partial)=-k(-\lambda-\mu-\partial,\mu+\partial)h(\mu).\label{eqkh}
\end{equation}
By comparing the degree of $\lambda$, we get the fact that the degree of $\partial$ is smaller than 1. Thus, \eqref{eqkh} become $k(\mu)k(-\lambda-\mu-\partial)=0 $ and then we get $k(\lambda,\partial)=0$. \eqref{lfd1} implies $h(\mu)D(\lambda+\mu,\partial)=0 $ and so we get at least one of $h(\mu)$ and $D(\lambda+\mu,\partial)$ to be zero.\par
To begin with, we assume $h(\mu)=0$ which means $l(a)_\lambda x$ and $r(a)_\lambda x$ are trivial and $Q=\mathbb{C}[\partial]x$ is a left-symmetric conformal algebra by \eqref{lfd10}. Thus, we know there are three cases $(1)P(\lambda,\partial)=0$,  $(2)P(\lambda,\partial)=1$ and $(3)P(\lambda,\partial)=\lambda+\partial+c$ by Proposition \ref{proplsca} and Theorem \ref{th4.3}. \par
We consider $P(\lambda,\partial)=0$ firstly. \eqref{lfd7} will reduce to $D(\mu,\lambda+\partial)D(\lambda,\partial)=D(\lambda,\mu+\partial)D(\mu,\partial)$. Therefore we get the degree of $\partial$ in $D(\lambda,\partial)$ is equal to 0 by comparing the degree of $\lambda$. In what follows, \eqref{lfd5} will reduce to
$$
T(\mu,-\lambda-\mu)T(-\lambda-\mu-\partial,\partial)-D(\mu)T(-\lambda-\mu-\partial,\partial)=-D(\mu)T(-\lambda-\mu-\partial,\mu+\partial),
$$
and then we can get the degree of $\partial$ in $T(\lambda,\partial)$ is equal to 0 by comparing the degree of $\lambda$. Thus, we get $T(\lambda,\partial)=0$ in that $T(\mu)T(-\lambda-\mu-\partial)=0$. Thus, this is $Case (a)$.\par

Next, we consider $P(\lambda,\partial)=1$. By \eqref{lfd7}, we also get the degree of $\partial$ in $D(\lambda,\partial)$ is equal to 0 and then by \eqref{lfd5} we get
\begin{equation*}
T(\mu,-\lambda-\mu)T(-\lambda-\mu-\partial,\partial)-D(\mu)T(-\lambda-\mu-\partial,\partial)-T(-\lambda-\partial,\partial)
=-D(\mu)T(-\lambda-\mu-\partial,\mu+\partial).
\end{equation*}
It is easy to get that the degree of $\partial$ in $T(\lambda,\partial)$ is equal to 0. Then we can get $T(\lambda,\partial)=1$ or $0$ from $T(\mu)T(-\lambda-\mu-\partial)=T(-\lambda-\partial)$. Therefore one has $Case(b)$ when $T(\lambda,\partial)=1$ and $Case(c)$ when $T(\lambda,\partial)=0$. \par
Finally, we consider $P(\lambda,\partial)=\lambda+\partial+c$. We can get
\begin{equation}
(\lambda-\mu)D(\lambda+\mu,\partial)=D(\mu,\lambda+\partial)D(\lambda,\partial)-D(\lambda,\mu+\partial)D(\mu,\partial),\label{dd}
\end{equation}
via \eqref{lfd7} and then one has that the degree of $\partial$ in $D(\lambda,\partial)$ is smaller than 2 by comparing the degree of $\lambda$. Then we can assume $D(\lambda,\partial)=d_1(\lambda)\partial+d_0(\lambda)$ and take it into \eqref{dd}. Hence
\begin{equation*}
(\lambda-\mu)(d_1(\lambda+\mu)\partial+d_0(\lambda+\mu))=(\lambda-\mu)\partial d_1(\mu)d_1(\lambda)+\lambda d_1(\mu)d_0(\lambda)-\mu d_1(\lambda)d_0(\mu),
\end{equation*}
and then we get $d_1(\lambda)=1\text{ or }0$ since $d_1(\lambda+\mu)=d_1(\mu)d_1(\lambda) $. If $d_1(\lambda)=0$, one has $d_0(\lambda)=0$ and then $D(\lambda,\partial)=0$. Hence \eqref{lfd5} simplify to
\begin{equation*}
T(\mu,-\lambda-\mu)T(-\lambda-\mu-\partial,\partial)=(\mu+\lambda+\partial+c)T(-\lambda-\partial,\partial),
\end{equation*}
and then we have $T(\lambda,\partial)=0$ by comparing the degree of $\partial$. Thus, this is $Case(d)$.
If $d_1(\lambda)=1$, one has $(\lambda-\mu)d_0(\lambda+\mu)=\lambda d_0(\lambda)-\mu d_0(\mu) $. Then we can deduce that the degree of $\lambda$ in $d_0(\lambda)$ is smaller than 2, so we write $d_0(\lambda)=\alpha\lambda+\beta$ and $D(\lambda,\partial)=\partial+\alpha\lambda+\beta$ where $\alpha,\beta\in\mathbb{C}$. In what follows, \eqref{lfd5} can be written as
\begin{align}\label{tdq}
&[T(\mu,-\lambda-\mu)-(-\lambda-\mu+\alpha\mu+\beta)]T(-\lambda-\mu-\partial,\partial)-(\mu+\lambda+\partial+c)T(-\lambda-\partial,\partial)\\
=&-(\partial+\alpha\mu+\beta)T(-\lambda-\mu-\partial,\mu+\partial).\nonumber
\end{align}
One has the degree of $\partial$ in $T(\lambda,\partial)$ is smaller than 2 by comparing the degree of $\lambda$ in \eqref{tdq}.
We have
\begin{align*}
&[(-\lambda-\mu) t_1(\mu)+t_0(\mu)+\lambda-(\alpha-1)\mu-\beta](t_1(-\lambda-\mu-\partial)\partial+t_0(-\lambda-\mu-\partial))\\
&-(\mu+\lambda+\partial+c)(t_1(-\lambda-\partial)\partial+t_0(-\lambda-\partial))=-(\partial+\alpha\mu+\beta)(\mu+\partial)t_1(-\lambda-\mu-\partial)\\
&-(\partial+\alpha\mu+\beta)t_0(-\lambda-\mu-\partial),
\end{align*}
 by assuming $T(\lambda,\partial)=t_1(\lambda)\partial+t_0(\lambda)$ and taking it into \eqref{tdq}.
 Then we consider the degree of $\lambda$ in $t_1(\lambda)$ and $t_0(\lambda)$. Let $deg_\lambda(t_1(\lambda))=k_1$ and $deg_\lambda(t_0(\lambda))=k_2 $. On the one hand, if $t_1(\lambda)=0$ then one has $ deg_\mu(t_0(\mu))<2$ by comparing the degree of $\mu$.
 Assuming $ t_0(\lambda)=t_2\lambda+t_3$, by comparing the coefficient of $\mu^2$ one obtains $t_2=0$ or $t_2=-1$. If $t_2=0$, $t_3=0$ or $t_3=c$. If $t_2=-1$, $t_3=c$. Similarly, if $t_0(\lambda)=0$ then one has the degree of $t_1(\mu)$ is smaller than 2 by comparing the degree of $\mu$.
 We should note that $deg_\mu(t_1(\mu))>0$ is impossible by comparing the coefficient of $\mu^3\partial$. Thus we get $ T(\lambda,\partial)=0$ with a few simple calculations.
 On the other hand, If $k_1>k_2$ then we get $k_1=0$ by comparing the coefficients of $\lambda^{k_1+1}\partial $ and if $k_2>k_1$, we  obtain $t_1(\mu)=0$ by comparing the coefficients of $\lambda^{k_2+1} $.
 If $k_1=k_2$ then we obtain $k_1=k_2=0$ by comparing the coefficient of $\mu^{2k_1+1}\partial$. In what follows, we obtain $t_1(\mu)=0$ by comparing the coefficient of $\lambda\partial$. Hence one has $Case(e)$ which $T(\lambda,\partial)=0$, $Case(f)$ which $T(\lambda,\partial)=c$ and $Case(g)$ which $T(\lambda,\partial)=-\lambda+c$.\par
 In the end, we suppose $D(\lambda,\partial)=0$ then we obtain $deg_\lambda(T(\lambda,\partial))=deg_\lambda(h(\lambda))$ by comparing the degree of $\lambda$ in \eqref{lfd3}. When $p(\lambda,\partial)=0$ one has $T(\lambda,\partial)=0$ by \eqref{lfd5}. This is $Case(h)$. When $p(\lambda,\partial)=1$ one has the degree of $\partial$ in $T(\lambda,\partial)$ is equal to zero by comparing the degree of $\lambda$ in \eqref{lfd5} and then by comparing the degree of $\mu$ one obtains $ T(\lambda,\partial)=1$ or $T(\lambda,\partial)=0$. If $T(\lambda,\partial)=0$, by \eqref{lfd6} one can get $T(\lambda,\partial)=h(\lambda)=0$ and then this is the special case in $Case(c)$ when $D(\lambda)=0$. If $T(\lambda,\partial)=1 $, one has $h(\lambda,\partial)\in\mathbb{C}$ since $h(\lambda)=-h(\lambda+\mu)$ in \eqref{lfd6} and then this is the $Case(i)$. When $p(\lambda,\partial)=\lambda+\partial+c$ one gets $T(\lambda,\partial)=0$ by comparing the degree of $\partial$ in \eqref{lfd5} and then $h(\lambda,\partial)=0$ from \eqref{lfd6}. We know this is $Case(d)$.

\end{proof}}

\section{Special cases of unified products and examples}\label{sec5}
In this section, we will introduce some important and interesting products of left-symmetric conformal algebras such as
crossed products and bicrossed products which are all special cases of unified products.
\subsection{Crossed products of left-symmetric conformal algebras }\label{cp}\

Let $R$ be a left-symmetric conformal algebra and $Q$ be a $\mathbb{C}[\partial]$-module. Let $\Omega(R,Q)= (\varphi,\psi,l,r,g_\lambda(\cdot,\cdot),\circ_\lambda)$ be an extending datum of $R$ by $Q$ where $l$ and $r$ are trivial. We denote this extending datum simply by $\Omega(R,Q)= (\varphi,\psi, g_\lambda(\cdot,\cdot), \circ_\lambda)$.
Then $\Omega(R,Q)= (\varphi,\psi, g_\lambda(\cdot,\cdot), \circ_\lambda)$ is a left-symmetric conformal extending structure of $R$ by $Q$ if and only if  $(Q,\circ_\lambda)$ is a left-symmetric conformal algebra and the following conditions are satisfied for all $a$, $b\in R$ and $x$, $y$, $z\in Q$:
\begin{align}
&(\varphi(x)_\lambda a-\psi(x)_{-\mu-\partial}a)_{\lambda+\mu}b=\varphi(x)_\lambda(a_\mu b)-a_\mu(\varphi(x)_\lambda b),\tag{C1}\label{C1}\\
&\psi(x)_{-\lambda-\mu-\partial}(a_\lambda b-b_\mu a)=a_\lambda(\psi(x)_{-\mu-\partial}b)-b_\mu(\psi(x)_{-\lambda-\partial}a),\tag{C2}\label{C2}\\
&\psi(y)_{-\lambda-\mu-\partial}(\psi(x)_{-\lambda-\partial}a-\varphi(x)_\mu a )-a_\lambda(g_\mu(x,y))
-\psi(x\circ_\mu y)_{-\lambda-\partial}a=-\varphi(x)_\mu(\psi(y)_{-\lambda-\partial}a),\tag{C3}\label{C3}\\
&(g_\lambda(x,y)-g_\mu(y,x))_{\lambda+\mu}a+\varphi(x\circ_\lambda y-y\circ_\mu x)_{\lambda+\mu}a=\varphi(x)_\lambda(\varphi(y)_\mu a)
-\varphi(y)_\mu(\varphi(x)_\lambda a),\tag{C4}\label{C4}\\
&\psi(z)_{-\lambda-\mu-\partial}g_\lambda(x,y)+g_{\lambda+\mu}(x\circ_\lambda y,z)-\varphi(x)_\lambda g_\mu(y,z)-g_\lambda(x,y\circ_\mu z)\tag{C5}\label{C5}\\
&=\psi(z)_{-\lambda-\mu-\partial}g_\mu(y,x)+g_{\lambda+\mu}(y\circ_\mu x,z)-\varphi(y)_\mu g_\lambda(x,z)-g_\mu(y,x\circ_\lambda z).\nonumber
\end{align}

We denote the associated unified product $R\natural Q$
 by $R\natural^g_{\varphi,\psi} Q $ and call it the {\bf crossed product} of $R$ and $Q$. The $\lambda$-products on $R\natural^g_{\varphi,\psi} Q $ are given
 by for all $a,b\in R$ and $x,y\in Q$:
 $$
 (a+x)_\lambda(b+y)=(a_\lambda b+\varphi (x)_\lambda b+\psi(y)_{-\lambda-\partial}a+g_\lambda(x,y))+x\circ_\lambda y.
 $$
It is obvious that $R$ is an ideal of $R\natural^g_{\varphi,\psi} Q $.
\begin{proposition}\label{crossed-product}
Let $R$ and $Q$ be two left-symmetric conformal algebras. Set $E=R\oplus Q$ where the direct sum is the sum of $\mathbb{C}[\partial]$-modules. If $E$ has a left-symmetric conformal algebra structure such that $R$ is an ideal of $E$,
then $E$ is isomorphic to a crossed product $R\natural^g_{\varphi,\psi} Q $ of $R$ and $Q$.
\end{proposition}
\begin{proof}
It is straightforward by Theorem \ref{th35}.
\end{proof}

Therefore, crossed products of left-symmetric conformal algebras are useful for investigating the $\mathbb{C}[\partial]$-split extension problem given in the introduction. By Proposition \ref{crossed-product}, any $E$ in the $\mathbb{C}[\partial]$-split extension problem is isomorphic to a crossed product $R\natural^g_{\varphi,\psi} Q $. Notice that all crossed products of $R$ and $Q$ satisfy the conditions in the $\mathbb{C}[\partial]$-split extension problem. Therefore, all left-symmetric conformal algebra structures on $E$ in the $\mathbb{C}[\partial]$-split extension problem can be described by all crossed products of $R$ and $Q$ up to isomorphism which stabilizes $R$ and co-stabilizes $Q$. By Lemma \ref{lemma31} and Theorem \ref{th3.12}, the $\mathbb{C}[\partial]$-split extension problem can be answered by $\mathcal{H}^2(Q,R) \cong \mathfrak{L}(R,Q)/\approx $ where in these left-symmetric conformal extending structures $l,r=0$ and $\circ_\lambda$ is the $\lambda$-product on $Q$, which is simply denoted by $\mathcal{HC}^2(Q,R)$.

In what follows, we consider the case when $R=\mathbb{C}[\partial]V$ is a left-symmetric conformal algebra which is free as a $\mathbb{C}[\partial]$-module and $Q=\mathbb{C}[\partial]x$ is a left-symmetric conformal algebra which is free of rank one as a $\mathbb{C}[\partial]$-module. By Theorem \ref{th4.3}, $\mathcal{HC}^2(Q,R)$ can be characterized by flag datums of $R$ with $h_\lambda(\cdot,\partial)=k_\lambda(\cdot,\partial)=0$ and a given $P(\lambda,\partial)$. Notice that for a crossed product of $R$ and $Q=\mathbb{C}[\partial]x$, $T_\lambda$ is a conformal semi-quasicentroid of $R$ in any flag datum of $R$.

\begin{proposition}\label{cross}
Let $R=\mathbb{C}[\partial]V$ be a left-symmetric conformal algebra which is free as a $\mathbb{C}[\partial]$-module and $Q=\mathbb{C}[\partial]x$ be a left-symmetric conformal algebra which is free of rank one as a $\mathbb{C}[\partial]$-module. Suppose that there does not exist any non-zero element $b$ such that $a_\lambda b=0$ for all $a\in R$.
Then $\mathcal{HC}^2(Q,R) \cong \mathcal{FLC}(R)/\approx $, where $ \approx$ is the equivalence relation on $\mathcal{FLC}(R) $ given by:
$$
(0,0,D_\lambda,T_\lambda,M(\lambda,\partial),P(\lambda,\partial))\approx(0,0,D'_\lambda,T'_\lambda,M'(\lambda,\partial),P(\lambda,\partial))
$$
if and only if there exists $\omega\in R$ such that $T_{\lambda}(a)= T'_{\lambda}(a)+a_{-\lambda-\partial} \omega$ and $D_\lambda(a)=D'_{\lambda}(a)+\omega_\lambda a$ for all $a\in V$. Moreover, if all conformal derivations of $R$ are zero, then $\mathcal{HC}^2_R(Q,R) \cong \mathcal{FLC}(R)/\approx  $, where $ \approx$ is the equivalence relation on $\mathcal{FLC}(R) $ given by:
$$
(0,0,D_\lambda,T_\lambda,M(\lambda,\partial),P(\lambda,\partial))\approx  (0,0,D'_\lambda,T'_\lambda,M'(\lambda,\partial),P(\lambda,\partial))
$$
if and only if $T_{\lambda}-T'_{\lambda}\in CQSInn(R)$.

\delete{(2) If the $\lambda$-products in $Q$ are non-trivial, then $\mathcal{HC}^2(Q,R)\cong \mathcal{FLC}(R)/\approx $, where $ \approx$ is the equivalence relation on $\mathcal{FLC}(R) $ given by:
$$
(0,0,D_\lambda,T_\lambda,M(\lambda,\partial),P(\lambda,\partial))\approx(0,0,D'_\lambda,T'_\lambda,M'(\lambda,\partial),P(\lambda,\partial))
$$
if and only if $T_{-\lambda-\partial}(a)= T'_{-\lambda-\partial}(a)+a_\lambda \omega$ and $D_\lambda(a)= D'_{\lambda}(a)+\omega_\lambda a$ for all $a\in V$. Moreover, if all conformal derivations of $R$ are zero, then $\mathcal{HC}^2(Q,R)\cong \mathcal{FLC}(R)/\approx $, where $ \equiv$ is the equivalence relation on $\mathcal{FLC}(R) $ given by:
$$
(0,0,D_\lambda,T_\lambda,M(\lambda,\partial),0)\approx(0,0,D'_\lambda,T'_\lambda,M'(\lambda,\partial),0)
$$
if and only if $T_{-\lambda-\partial}-T'_{-\lambda-\partial}\in CQSInn(R)$.}
\end{proposition}
\begin{proof}
By Theorem \ref{th4.3}, for $\mathcal{HC}^2(Q,R)$, we only need to show that if there exists $\omega\in R$ such that $T_{\lambda}(a)= T'_{\lambda}(a)+a_{-\lambda-\partial} \omega$ and $D_\lambda(a)=D'_{\lambda}(a)+\omega_\lambda a$ for all $a\in V$, then $(0,0, D_\lambda,T_\lambda,M(\lambda,\partial),$ $P(\lambda,\partial))\approx (0,0,D'_\lambda,T'_\lambda, M'(\lambda,\partial),P(\lambda,\partial)) $. \delete{By Proposition \ref{proplsca}, we get that $P(\lambda,\partial)=0$, $P(\lambda,\partial)=1$ or $P(\lambda,\partial)=\lambda+\partial+c$ for some $c\in \mathbb{C}$.}
Taking $T_{\lambda}(a)= T'_{\lambda}(a)+a_{-\lambda-\partial} \omega$ and $D_\lambda(a)=D'_{\lambda}(a)+\omega_\lambda a$ for all $a\in V$ into \eqref{lfd5}, we get
\begin{align*}
&T'_{-\lambda-\mu-\partial}(T'_{-\lambda-\partial}(a))-T'_{-\lambda-\mu-\partial}(D'_\mu(a))+D'_\mu(T'_{-\lambda-\partial}(a))
+(T'_{-\lambda-\partial}(a))_{\lambda+\mu}\omega\\
&+T'_{-\lambda-\mu-\partial}(a_\lambda\omega)-(D'_\mu(a))_{\lambda+\mu}\omega-T'_{-\lambda-\mu-\partial}(\omega_\mu a)+\omega_\mu(T'_{-\lambda-\partial}(a))+D'_\mu(a_\lambda\omega)\\
&+(a_\lambda \omega)_{\lambda+\mu}\omega-(\omega_\mu a)_{\lambda+\mu}\omega+\omega_\mu(a_\lambda\omega)=a_\lambda(M(\mu,\partial))+P(\mu,\lambda+\partial)T'_{-\lambda-\partial}(a)+P(\mu,\lambda+\partial)a_\lambda \omega.
\end{align*}
Then by the left-symmetry identity and $T'_{\lambda}$ is a conformal semi-quasicentroid, we get
$$
a_\lambda(M(\mu,\partial))=a_\lambda(\omega_\mu\omega)+a_\lambda(M'(\mu,\partial))+ a_\lambda(D'_\mu(\omega))+ a_\lambda(T'_{-\mu-\partial}(\omega))-a_\lambda(P(\mu,\partial)\omega).
$$
Since there does not exist any non-zero element $b$ such that $a_\lambda b=0$ for all $a\in R$, one has $M(\mu,\partial)=\omega_\mu\omega+M'(\mu,\partial)+D'_\mu(\omega)+T'_{-\mu-\partial}(\omega)-P(\mu,\partial)\omega$. Then we get
$$
(0,0,D_\lambda,T_\lambda,M(\lambda,\partial),P(\lambda,\partial))\approx (0,0, D'_\lambda,T'_\lambda,M'(\lambda,\partial),P(\lambda,\partial)) .
$$

Suppose that all conformal derivations of $R$ are zero. By Theorem \ref{th4.3} and the discussion above, we only need to show that if $T_{\mu}(a)- T'_{\mu}(a)= a_{-\mu-\partial}\omega$ for all $a\in V$ and some $\omega\in R$, then $D_\lambda(a)=D'_{\lambda}(a)+\omega_\lambda a$ for all $a\in V$.
Taking $T_{\mu}(a)- T'_{\mu}(a)= a_{-\mu-\partial}\omega$ for all $a\in V$ into \eqref{lfd1}, we get
\begin{align*} D_\lambda(a)_{\lambda+\mu}b-T'_{\lambda }(a)_{\lambda+\mu}b-(a_\mu\omega)_{\lambda+\mu}b=D_\lambda(a_\mu b)-a_\mu D_\lambda (b).
\end{align*}
Then by the left-symmetry identity, we get
\begin{align*}
((D_\lambda -\omega_\lambda) a)_{\lambda+\mu}b-T'_\lambda (a)_{\lambda+\mu}b=(D_\lambda-\omega_\lambda)(a_\mu b)-a_\mu((D_\lambda -\omega_\lambda)b).
\end{align*}
Notice that
\begin{eqnarray*}
(D'_\lambda(a)-T'_\lambda(a))_{\lambda+\mu}b=D'_\lambda(a_\mu b)-a_\mu(D'_\lambda(b)).
\end{eqnarray*}
Then we have
\begin{eqnarray*}
(D'_\lambda-D_\lambda+\omega_\lambda)(a_\mu b)=((D'_\lambda-D_\lambda+\omega_\lambda) a)_{\lambda+\mu}b+a_\mu ((D'_\lambda-D_\lambda+\omega_\lambda) b).
\end{eqnarray*}
Therefore, $D'_\lambda-D_\lambda+\omega_\lambda$ is a conformal derivation of $R$. Since all conformal derivations of $R$ is zero, we get $D'_\lambda=D_\lambda-\omega_\lambda$. Then the proof is completed.

\end{proof}
\begin{remark}
By Proposition \ref{proplsca}, for $P(\lambda,\partial)$ in Proposition \ref{cross}, there are three probabilities, i.e. $P(\lambda,\partial)=0$, $P(\lambda,\partial)=c_1$ or $P(\lambda,\partial)=\partial+\lambda+c_2$ for some $c_1\in \mathbb{C}\backslash\{0\}$ and $c_2\in \mathbb{C}$.
\end{remark}
Finally, we present an example to compute $\mathcal{HC}^2(Q, R)$.
\begin{example}
Let $R=\mathbb{C}[\partial]L\oplus\mathbb{C}[\partial]W$ be a left-symmetric conformal algebra with the $\lambda$-products as follows:
\begin{equation}
L_\lambda L=0,\,\,\,\,\, L_\lambda W=W_\lambda L=L, \,\,\,\,\, W_\lambda W=W
\end{equation}
and $Q=\mathbb{C}[\partial]x $ be a left-symmetric conformal algebra with the trivial $\lambda$-products. 
Assume $D_\lambda L=D_1(\lambda,\partial)L+D_2(\lambda,\partial)W$,  $D_\lambda W=d_1(\lambda,\partial)L+d_2(\lambda,\partial)W$, $T_\lambda L=T_1(\lambda,\partial)L+T_2(\lambda,\partial)W$ and $T_\lambda W=t_1(\lambda,\partial)L+t_2(\lambda,\partial)W$, where $D_i(\lambda,\partial)$, $d_i(\lambda,\partial)$, $T_i(\lambda,\partial)$, and $t_i(\lambda,\partial)\in \mathbb{C}[\lambda,\partial]$ for $i=1$, $2$. Then by \eqref{lfd3} we get
\begin{align}
&\label{ee1}T_2(-\lambda-\mu-\partial,\lambda+\partial)L=T_2(-\lambda-\mu-\partial,\mu+\partial)L,\\
&t_2(-\lambda-\mu-\partial,\lambda+\partial)L=T_1(-\lambda-\mu-\partial,\mu+\partial)L+T_2(-\lambda-\mu-\partial,\mu+\partial)W, \label{e1}\\
&t_1(-\lambda-\mu-\partial,\lambda+\partial)L+t_2(-\lambda-\mu-\partial,\lambda+\partial)W=t_1(-\lambda-\mu-\partial,\mu+\partial)L \label{e2}\\
&+t_2(-\lambda-\mu-\partial,\mu+\partial)W.\nonumber
\end{align}
We get $T_2(-\lambda-\mu-\partial,\mu+\partial)=0$ by comparing the coefficient of $W$ in \eqref{e1} and the degrees of $\partial$ in $t_1(\lambda,\partial)$ and $t_2(\lambda,\partial)$ are equal to $0$ by comparing the coefficients of $L$ and $W$ in \eqref{e2} respectively. Therefore we set $t_1(\lambda,\partial)=t_1(\lambda)$ and $t_2(\lambda,\partial)=t_2(\lambda)$ for some $t_1(\lambda)$ and $t_2(\lambda)\in \mathbb{C}[\lambda]$. Then by comparing the coefficient of $L$ in (\ref{e1}), we get $T_1(\lambda,\mu+\partial)=T_1(\lambda)=t_2(\lambda)$, where $T_1(\lambda)\in \mathbb{C}[\lambda]$.

By \eqref{lfd1}, we get
\begin{align}
&D_2(\lambda,-\lambda-\mu)L=-D_2(\lambda,\mu+\partial)L,\label{D0}  \\
&D_1(\lambda,-\lambda-\mu)L+D_2(\lambda,-\lambda-\mu)W-T_1(\lambda)L=D_1(\lambda,\partial)L+D_2(\lambda,\partial)W-d_2(\lambda,\mu+\partial)L,\label{e3}\\
&d_1(\lambda,-\lambda-\mu)L+d_2(\lambda,-\lambda-\mu)W-t_1(\lambda)L-t_2(\lambda)W\label{e4}\\
&=d_1(\lambda,\partial)L+d_2(\lambda,\partial)W-d_1(\lambda,\mu+\partial)L-d_2(\lambda,\mu+\partial)W,\nonumber\\
&d_2(\lambda,-\lambda-\mu)L-t_2(\lambda)L=D_1(\lambda,\partial)L+D_2(\lambda,\partial)W-D_1(\lambda,\mu+\partial)L-D_2(\lambda,\mu+\partial)W.\label{e5}\
\end{align}
\eqref{D0} implies $D_2(\lambda,\partial)=0$. It follows by comparing the degree of $\mu\partial$ in the coefficient of $L$ in \eqref{e3} and \eqref{e4} that  the degrees of $\partial$ in $d_2(\lambda,\partial)$, $D_1(\lambda,\partial)$ and $d_1(\lambda,\partial)$ are smaller than 2.  Assume $d_2(\lambda,\partial)=h_1(\lambda)\partial +h_0(\lambda)$ where $h_0(\lambda)$, $h_1(\lambda)\in \mathbb{C}[\lambda]$ and take it into \eqref{e4}. Then we get $d_2(\lambda,-\lambda)=t_2(\lambda)$. Similarly, we can get
$d_1(\lambda,-\lambda)=t_1(\lambda)$. Assume $D_1(\lambda,\partial)=k_1(\lambda)\partial+k_0(\lambda)$ where $k_0(\lambda)$, $k_1(\lambda)\in \mathbb{C}[\lambda]$ and take it into \eqref{e5}. It follows that $h_1(\lambda)=k_1(\lambda)$. Assume $M(\lambda,\partial)=q_1(\lambda,\partial)L+q_2(\lambda,\partial)W$, where $q_1(\lambda,\partial)$, $q_1(\lambda,\partial)\in \mathbb{C}[\lambda,\partial]$.
By  \eqref{lfd7}, we get
\begin{align}
&q_2(\lambda,-\lambda-\mu)-q_2(\mu,-\lambda-\mu)=D_1(\mu,\lambda+\partial)D_1(\lambda,\partial)-D_1(\lambda,\mu+\partial)D_1(\mu,\partial),  \label{e6}\\
&q_1(\lambda,-\lambda-\mu)-q_1(\mu,-\lambda-\mu)=d_1(\mu,\lambda+\partial)D_1(\lambda,\partial)+d_2(\mu,\lambda+\partial)d_1(\lambda,\partial)\label{e7}\\
&-d_1(\lambda,\mu+\partial)D_1(\mu,\partial)-d_2(\lambda,\mu+\partial)d_1(\mu,\partial).\nonumber
\end{align}
Since
$$
D_1(\mu,\lambda+\partial)D_1(\lambda,\partial)-D_1(\lambda,\mu+\partial)D_1(\mu,\partial)
=h_1(\lambda)h_1(\mu)(\lambda-\mu)\partial+h_1(\mu)k_0(\lambda)\lambda-h_1(\lambda)k_0(\mu)\mu,
$$
then the degree of $\partial$ in $D_1(\lambda,\partial)$ is equal to 0 by comparing the degree of $\partial$ in \eqref{e6}.
Therefore $h_1(\lambda)=0$ and $q_2(\lambda,-\lambda-\mu)-q_2(\mu,-\lambda-\mu)=0$. Then the degree of $\lambda$ in $q_2(\lambda,\partial)$ is equal to $0$.
On the other hand, by \eqref{lfd5}, we get
\begin{align}
&T_1(\mu)T_1(-\lambda-\mu-\partial)-D_1(\mu)T_1(-\lambda-\mu-\partial)-q_2(\mu,\lambda+\partial)=-T_1(-\lambda-\mu-\partial)D_1(\mu),\label{e8}\\
&t_1(\mu)T_1(-\lambda-\mu-\partial)+t_2(\mu)t_1(-\lambda-\mu-\partial)-d_1(\mu,-\lambda-\mu)T_1(-\lambda-\mu-\partial)\label{e9}\\
&-d_2(\mu)t_1(-\lambda-\mu-\partial)-q_1(\mu,\lambda+\partial)=-t_1(-\lambda-\mu-\partial)D_1(\mu)-t_2(-\lambda-\mu-\partial)d_1(\mu,\partial).\nonumber
\end{align}
By comparing the degree of $\mu$ in \eqref{e8}, we get $ t_2(\lambda,\partial)=d_2(\lambda,\partial)=T_1(\lambda,\partial)=h_0(\lambda)\in \mathbb{C}$ and denote it by $h$. Then $q_2(\lambda,\partial)=h^2$. Hence we obtain $q_1(\mu,\lambda+\partial)=d_1(-\lambda-\mu-\partial,\lambda+\mu+\partial)k_0(\mu)+hd_1(\mu,\lambda+\partial)$
and \eqref{lfd5} naturally holds. Therefore, we have
\begin{align*}
&D_\lambda L=k_0(\lambda)L,  \,\,\,\,\,\,D_\lambda W=d_1(\lambda,\partial)L+hW,\\
&T_\lambda L=hL,                     \,\,\,\,\,\,T_\lambda W=d_1(\lambda,-\lambda)L+hW.
\end{align*}
Assume $d_1(\lambda,\partial)=p_1(\lambda)\partial+p_0(\lambda)$, where $p_0(\lambda)$, $p_1(\lambda)\in \mathbb{C}[\lambda]$. The flag datum $(D_\lambda,T_\lambda,M(\lambda,\partial),$ $P(\lambda,\partial))$ is determined by $k_0(\lambda)$, $p_1(\lambda)$, $p_0(\lambda)$ and $h$. Therefore, we denote this flag datum by $(k_0(\lambda),p_1(\lambda),p_0(\lambda),h)$.
Assume $\omega=f(\partial)L+g(\partial)W$ in Theorem \ref{th4.3}. Then by Theorem \ref{th4.3}, $(k_0(\lambda),p_1(\lambda),p_0(\lambda),h)\approx (k'_0(\lambda),p'_1(\lambda),p'_0(\lambda),h')$ if and only if there exist $f(\partial)$ and $g(\partial)\in \mathbb{C}[\partial]$ such that
\begin{align*}
&p_1(\lambda)=p'_1(\lambda), \\
&p_0(\lambda)=p'_0(\lambda)+f(-\lambda),\\
&h=h'+g(-\lambda),\\
&k_0(\lambda)=k'_0(\lambda)+g(-\lambda),\\
&h=h'+g(\lambda+\partial).
\end{align*}
Therefore, $g(\partial)=g\in \mathbb{C}$.
 Let $f(\partial)=p_0(-\partial)$ and $g(\partial)=h$. Then we have $(k_0(\lambda),p_1(\lambda),p_0(\lambda),h)\approx
 (k(\lambda)=k_0(\lambda)-h,p_1(\lambda),0,0)$. Notice that $(k(\lambda),p_1(\lambda),0,0)\approx (k'(\lambda),p'_1(\lambda),0,0) $ if and only if $k(\lambda)=k'(\lambda)$ and $p_1(\lambda)=p'_1(\lambda)$.
 Hence $\mathcal{HC}^2(Q,R)$ can be described by all flag datums of the form  $(k(\lambda),p_1(\lambda),0,0)$, where $k(\lambda)$ and $p_1(\lambda)\in \mathbb{C}[\lambda]$.
\delete{
In the end,
\begin{align*}
&D_\lambda a=k_0(\lambda)a,  \,\,\,\,\, D_\lambda b=d_1(\lambda,\partial)a+h b,\\
&T_\lambda a=h a, \,\,\,\,\,  T_\lambda b=d_1(\lambda,-\lambda)a+h b,\\
&M(\lambda,\partial)=q_1(\lambda,\partial)a+h^2 b.
\end{align*
If $d_1(\lambda,\partial)=0$ then $q_1(\lambda,\partial)=0$ . By Theorem \ref{th4.3}, $(0,0,D_\lambda,T_\lambda,M(\lambda,\partial),0)\equiv (0,0,D'_\lambda,0,0,0) $ if and only if
there exists $\beta\in\mathbb{C}\setminus\{0\}$ such that $D'_\lambda a=\frac{1}{\beta}(k_0(\lambda)-h)a $. If $d_1(\lambda,\partial)\neq0$ then}
by Theorem \ref{th4.3} we can get $(0,0,D_\lambda,T_\lambda,M(\lambda,\partial),0)\equiv (0,0,D'_\lambda,T'_\lambda,M'(\lambda,\partial),0) $  if and only if there exists $\omega=h b$ and $\beta\in \mathbb{C}\setminus\{0\}$ such that the following equalities are satisfied:
\begin{align*}
&\beta D_\lambda'a=(k_0(\lambda)-h)a,  \,\,\,\,\,\,\beta D'_\lambda b=d_1(\lambda,\partial)a,\\
&\beta T_\lambda'a=0,                     \,\,\,\,\,\,\beta T'_\lambda b=d_1(\lambda,-\lambda)a,\\
&\beta^2M'(\lambda,\partial)=(k_0(\lambda)-h)d_1(-\lambda-\partial,\lambda+\partial)a,
\end{align*}
where  $k_0\in\mathbb{C}[\lambda]$ and the degree of $\partial$ in $d_1(\lambda,\partial)$ is smaller than 1. Obviously any two kinds of flag datums which the degrees of $\lambda$ in $k_0(\lambda)$ or $d_1(\lambda)$  are distinct are not equivalent to each other by Theorem \ref{th4.3}.
Hence $\mathcal{HC}2_R(Q,R)$ can be described by the infinite kinds of flag datums above under the corresponding equivalences.}
\delete{
by \eqref{C5} we get
\begin{align}
&d_2(q_1(\lambda,-\lambda-\mu) -q_1(\mu,-\lambda-\mu))=q_1(\mu,\lambda+\partial)D_1(\lambda)-q_1(\lambda,\mu+\partial)D_1(\mu)\\
&+d_2^2(d_1(\lambda,\partial) -d_1(\mu,\partial))\nonumber
\end{align}}
\end{example}

\subsection{Bicrossed Products of Left-symmetric Conformal Algebras }\label{bp}\

Let $\Omega(R,Q)= (\varphi,\psi,l,r,g_\lambda(\cdot,\cdot), \circ_\lambda)$ be an extending datum of left-symmetric
conformal algebra $R$ by a $\mathbb{C}[\partial]$-module $Q$ where $g_\lambda(\cdot,\cdot) $ is  trivial.
Denote this extending datum simply by $(\varphi,\psi,l,r,\circ_\lambda)$.
Then $\Omega(R,Q)= (\varphi,\psi,l,r,\circ_\lambda)$ is a left-symmetric conformal extending structure of $R$ by $Q$ if and only if  $(Q, \circ_\lambda)$ is a left-symmetric conformal algebra and the following conditions are satisfied:\\
(1) $R$ is a $Q$-bimodule under $\varphi$, $\psi$ : $Q\rightarrow Cend(R)$.\\
(2) $Q$ is an $R$-bimodule under $l$, $r$ : $R\rightarrow Cend(Q) $.\\
(3) \eqref{LC1}, \eqref{LC3}, \eqref{LC6} and \eqref{LC8} hold.

 The associated unified product $R\natural Q$ denoted by
 $R\bowtie^{\varphi,\psi}_{l,r} Q$ is called the \text{\bf bicrossed product} of $R$ and $Q$. The $\lambda$-products on $R\bowtie^{\varphi,\psi}_{l,r} Q$
are given by for all $a,b\in R$ and $x,y \in Q$ as follows.
$$
(a+x)_\lambda(b+y)
=(a_\lambda b+\varphi (x)_\lambda b+\psi(y)_{-\lambda-\partial}a)+(x\circ_\lambda y+l(a)_\lambda y+r(b)_{-\lambda-\partial} x).
$$
Notice that $R$ and $Q$ are both subalgebras of $R\bowtie^{\varphi,\psi}_{l,r} Q$.

\begin{proposition}\label{prop5.4}
Let $R$ and $Q$ be two left-symmetric conformal algebras. Set $E=R\oplus Q$ where the direct sum is the sum of $\mathbb{C}[\partial]$-modules.
If $E$ is a left-symmetric conformal algebra such that $R$ and $Q$ are two subalgebras of $E$. Then $E$ is isomorphic
to a bicrossed product $R\bowtie^{\varphi,\psi}_{l,r} Q$ of left-symmetric conformal algebras $R$ and $Q$.
\end{proposition}
\begin{proof}
It is straightforward by Theorem \ref{th35}
\end{proof}

In fact, the bicrossed product of left-symmetric conformal algebras is useful for investigating the problem that describe and classify all left-symmetric conformal algebra structures on
$E=R\oplus Q$ such that $R$ and $Q$ are two subalgebras of $E$ up to isomorphism which stabilizes $R$ and co-stabilizes $Q$. By Proposition  \ref{prop5.4} and the general theory developed in Section 3,
this problem can be solved by  $\mathcal{H}^2(Q,R)=\mathfrak{L}(R,Q)/\approx $ where in these left-symmetric conformal extending structures $g_\lambda(\cdot,\cdot)=0$ and $\circ_\lambda$ is the $\lambda$-product on $Q$. For convenience we denote it by $\mathcal{HB}^2(Q,R)$. In particular, when $Q=\mathbb{C}[\partial]x$ and $R$ is free as a $\mathbb{C}[\partial]$-module, $\mathcal{HB}^2(Q,R) $ can be characterized by flag datums of $R$ where $M(\lambda,\partial)=0$ by Theorem \ref{th4.3}.

In the end, we give an example to compute $\mathcal{HB}^2(Q,R)$.
\begin{example}
Let $R=\mathbb{C}[\partial]L$ be a left-symmetric conformal algebra with the $\lambda$-product defined by $L_\lambda L=(\lambda+\partial+c)L$
where $c\in\mathbb{C}$ and $Q=\mathbb{C}[\partial]W$ be a left-symmetric conformal algebra with the $\lambda $-product defined by $W_\lambda W=(\lambda+\partial+\xi)W$ where $\xi\in\mathbb{C}$.

Denote by $h_\lambda(L,\partial)=h(\lambda,\partial)$, $k_\lambda(L,\partial)=k(\lambda,\partial)$, $D_\lambda(L)=D(\lambda,\partial)L$ and $T_\lambda(L)=T(\lambda,\partial)L$ where $h(\lambda,\partial), k(\lambda,\partial), D(\lambda,\partial)$ and $T(\lambda,\partial)\in \mathbb{C}[\lambda,\partial]$. Notice that  the flay datum $(h_\lambda(\cdot, \partial), k_\lambda(\cdot, \partial),$ $D_\lambda, T_\lambda, 0, P(\lambda,\partial))$ in this case is determined by $h(\lambda,\partial)$, $k(\lambda,\partial)$, $D(\lambda,\partial)$ and $T(\lambda,\partial)$, where $P(\lambda,\partial)=\lambda+\partial+\xi$. We denote it simply by $(h(\lambda,\partial), k(\lambda,\partial), D(\lambda,\partial), T(\lambda,\partial), \lambda+\partial+\xi)$.

Since $L_\lambda L=(\lambda+\partial+c)L$, we get that the degree of $\partial$ in $h(\lambda,\partial)$ is smaller than 2 by comparing the degree of $\lambda$ in \eqref{lfd4}. Assume $h(\lambda,\partial)=h_1(\lambda)\partial+h_2(\lambda)$ where $h_1(\lambda)$ and $h_2(\lambda)\in \mathbb{C}[\lambda]$. Take it into $(\lambda-\mu)h(\lambda+\mu,\partial)=h(\mu,\lambda+\partial)h(\lambda,\partial)-h(\lambda,\mu+\partial)h(\mu,\partial) $. Since $h_1(\lambda+\mu)=h_1(\lambda)h_1(\mu)$, we get that $h_1(\lambda)$ is equal to $1$ or $0$. If $h_1(\lambda)=1$, we obtain that $h_0(\lambda)=\alpha\lambda+\beta$ where $\alpha,\beta\in\mathbb{C}$ and $h(\lambda,\partial)=\partial+\alpha\lambda+\beta$. If $h_1(\lambda)=0$, we obtain that $h(\lambda,\partial)=0$. Similarly, by \eqref{lfd7}, one gets that $D(\lambda,\partial)$ is equal to $0$ or $\partial+\gamma\lambda+\delta$ where $\gamma,\delta\in\mathbb{C}$.

\textbf{Case 1:}  $h(\lambda,\partial)=0$.\\
 Then \eqref{lfd2} becomes $k(\mu,-\lambda-\mu)k(-\lambda-\mu-\partial,\partial)=(\lambda+\mu+\partial+c)k(-\lambda-\partial,\partial) $. Therefore one has $k(\lambda,\partial)=0$ by comparing the degree of $\partial$. If $D(\lambda,\partial)=0$, we get $T(\lambda,\partial)=0$ by \eqref{lfd1}.
On the other hand, if $D(\lambda,\partial)=\partial+\gamma\lambda+\delta$, then \eqref{lfd1} becomes
\begin{align*}
-\lambda^2+((\gamma-2)\lambda+\delta)(\mu+\partial)+c(\gamma-1)\lambda+c\delta=(\lambda+\mu+\partial+c)T(\lambda,-\lambda-\mu).
\end{align*}
The degree of $\partial$ in $T(\lambda,\partial)$ is equal to $0$ by comparing the degree of $\mu\partial$ and then the degree of $\lambda$ in $T(\lambda,\partial)$ is equal to $1$ by comparing the degree of $\lambda$. Therefore $T(\lambda,\partial)=(\gamma-2)\lambda+\delta$ by comparing the coefficient of $\partial$. Finally we get $\gamma=1$ and $\delta=c$ by comparing the coefficients of $\lambda^2$ and $\lambda$. It follows from \eqref{lfd5} that $\xi=c$. If $D_\lambda\neq 0$, then $(0,0,\partial+\lambda+c,-\lambda+c,\lambda+\partial+c)\approx (0,0,0,0,\lambda+\partial+c)$ by setting $\omega=L$ in  Theorem \ref{th4.3}. Therefore, in this case, there is only one equivalence class of flag datums, i.e. $(0,0,0,0,\lambda+\partial+c)$.

\textbf{Case 2:} $h(\lambda,\partial)=\partial+\alpha\lambda+\beta$.\\
Then we obtain the degree of $\partial$ in $k(\lambda,\partial)$ is smaller than $2$ by comparing the degree of $\lambda$ in \eqref{lfd2}. Assume $k(\lambda,\partial)=k_1(\lambda)\partial+k_0(\lambda)$ where $k_1(\lambda)$, $k_0(\lambda)\in \mathbb{C}[\lambda]$ and take it into \eqref{lfd2}. Then one has
\begin{eqnarray}\label{e22}
&&-k_1(\mu)k_1(-\lambda-\mu-\partial)\lambda\partial-k_1(\mu)k_1(-\lambda-\mu-\partial)\mu\partial+k_0(\mu)k_1(-\lambda-\mu-\partial)\partial\nonumber\\
&&-k_1(\mu)k_0(-\lambda-\mu-\partial)\lambda-k_1(\mu)k_0(-\lambda-\mu-\partial)\mu+k_0(\mu)k_0(-\lambda-\mu-\partial)\nonumber\\
&&+(\lambda+\mu+\partial)(k_1(-\lambda-\mu-\partial)\partial+k_0(-\lambda-\mu-\partial))+(\partial+\alpha\mu+\beta)k_1(-\lambda-\mu-\partial)\mu\nonumber\\
&=&(\lambda+\mu+\partial+c)(k_1(-\lambda-\partial)\partial+k_0(-\lambda-\partial)).
\end{eqnarray}
If $k_0(\lambda)=0$ and $k_1(\lambda)\neq 0$, we get that $k_1(\lambda)=0$ by comparing the coefficient of $\lambda^{m+1}\partial$ in \eqref{e22} where $m$ is the degree of $\lambda$ in $k_1(\lambda)$. Therefore, if $k_0(\lambda)=0$, we have $k_1(\lambda)=0$. If $k_1(\mu)=0$, one obtains that the degree of $\mu$ in $k_0(\mu)$ is smaller than $2$ by comparing the degree of $\mu$. Assume $k_0(\mu)=e\mu+f$ where $e,f\in\mathbb{C}$. We get that $e$ is equal to $-1$ or $0$ by comparing the coefficient of $\mu^2$. If $e=0$, one gets that $f$ is equal to $0$ or $c$. If $e=-1$, one has that $ f$ is equal to $c$.
Assume that both $k_1(\lambda)$ and $k_0(\lambda)$ are not equal to $0$ and the highest degrees of $\lambda$ in $k_1(\lambda)$ and $k_0(\lambda)$ are equal to $m$  and $n$ respectively. If $m\geq n$, one has $k_1(\mu)=0$ by comparing the coefficient of $\lambda^{m+1}\partial$. If $m<n $, one has $k_1(\mu)=0$ by comparing the coefficient of $\lambda^{n+1}$. Therefore, there are three cases for $k(\lambda,\partial)$. \\
\textbf{Subcase 1:} $k(\lambda,\partial)=0$. \\
If $D(\lambda,\partial)=0$, we have $T(\lambda,\partial)=0$ by \eqref{lfd1}, which is impossible by comparing the coefficient of $\lambda^2$ in \eqref{lfd6}. Therefore, we get $D(\lambda,\partial)=\partial+\gamma\lambda+\delta$.
Taking $D(\lambda,\partial)=\partial+\gamma\lambda+\delta $ into \eqref{lfd1}, we obtain
\begin{eqnarray}
&&(\gamma-1)\lambda^2+(1-\alpha)\gamma\mu^2+(\gamma-1)\lambda\partial+(\delta-\beta)\partial+(3\gamma-\alpha\gamma-2)\lambda\mu\nonumber\\
&&+(\delta-c+c\gamma-\beta\gamma)\lambda+(2\delta-\beta\gamma-\alpha\delta)\mu+(c-\beta)\delta+(1-\alpha)\mu\partial  \nonumber\\
\label{e2.1}&=&(\partial+\lambda+\mu+c)T(\lambda,-\lambda-\mu).
\end{eqnarray}
Then the degree of $\partial$ in $T(\lambda,\partial)$ is smaller than 2 by comparing the degree of $\mu$ in \eqref{e2.1}. Therefore, assume $T(\lambda,\partial)=t_1(\lambda)\partial+t_0(\lambda)$, where $t_0(\lambda)$, $t_1(\lambda)\in \mathbb{C}[\lambda]$. Comparing the coefficient of $\mu^2$ in \eqref{e2.1}, we obtain that the degree of $\lambda$ in $t_1(\lambda)$ is equal to $0$. Set $t_1(\lambda)=t_1$ where $t_1\in \mathbb{C}$. Taking $T(\lambda,\partial)=t_1\partial+t_0(\lambda)$ into \eqref{lfd3}, we get
\begin{align}\label{e2.2}
&(\lambda^2+\lambda\partial+c\lambda-\mu^2-\mu\partial-c\mu)t_1=(\mu\partial+\partial^2+\alpha\lambda\partial+\beta\partial)t_1
+(\mu+\partial+\alpha\lambda+\beta)t_0(-\mu-\partial)\\
&-(\lambda\partial+\partial^2+\alpha\mu\partial+\beta\partial)t_1
-(\lambda+\partial+\alpha\mu+\beta)t_0(-\lambda-\partial).\nonumber
\end{align}

If $t_1\neq0$, then $t_0(\lambda)=t_1\lambda+t$  by comparing the coefficient of $\lambda^2$ in \eqref{e2.2}, where $t_1$, $t\in\mathbb{C}$. By comparing the degree of $\mu$ in \eqref{lfd5}, we get $t_1=0$, which contradicts with our assumption.

If $t_1=0$, by \eqref{e2.2}, we get $T(\lambda,\partial)=t_2\lambda+t_3$ where $t_2,t_3\in\mathbb{C}$. Comparing the coefficients of $\lambda^2$, $\mu^2$, $\partial$ and $\lambda$ and constant term in \eqref{e2.1} and \eqref{lfd5}, we get that $t_2=0$, $\gamma=\alpha=1$, $\beta=\delta-t_3$, $t_3^2=t_3\xi$ and  $\beta(\delta-c)=0$. It follows  that $\beta(t_3-\delta+\xi)=0$ by comparing the constant term in \eqref{lfd6}. Therefore, if $\beta=0$, then $t_3$ is equal to $\delta$ which is equal to $\xi$ or $0$. If $\beta\neq0$, then $\delta=c$, $\beta=\xi$, and $t_3=c-\xi$ which is equal to $\xi$ or $0$.

Therefore, we have the following results in this case.

When $\xi=c=0$, any flag datum is equal to $(\lambda+\partial, 0, \lambda+\partial, 0, \lambda+\partial)$.

When $\xi=c\neq 0$, any flag datum is equal to one of the following forms: $(\lambda+\partial,0,\lambda+\partial+\xi,\xi,\lambda+\partial+\xi) $, $(\lambda+\partial,0,\lambda+\partial,0,\lambda+\partial+\xi )$ or $(\lambda+\partial+\xi,0,\lambda+\partial+\xi,0,\lambda+\partial+\xi)$.

When $\xi=0$ and $c\neq 0$, any flag datum is equal to $(\lambda+\partial,0,\lambda+\partial,0,\lambda+\partial)$.

When $\xi\neq0 $ and $c=0 $,  any flag datum is equal to one of the following forms: $(\lambda+\partial,0,\lambda+\partial+\xi,\xi,\lambda+\partial+\xi)$ or $(\lambda+\partial,0,\lambda+\partial,0,\lambda+\partial+\xi) $.

When $c=2\xi\neq0 $, any flag datum is equal to one of the following forms:
$(\lambda+\partial,0,\lambda+\partial+\xi,\xi,\lambda+\partial+\xi)$, $(\lambda+\partial,0,\lambda+\partial,0,\lambda+\partial+\xi) $ or $(\lambda+\partial+\xi,0,\lambda+\partial+2\xi,\xi,\lambda+\partial+\xi)$.

When $c\neq \xi,2\xi$ and $c,\xi$ are not equal to $0$, any flag datum is equal to one of the following forms: $(\lambda+\partial,0,\lambda+\partial+\xi,\xi,\lambda+\partial+\xi)$ or $(\lambda+\partial,0,\lambda+\partial,0,\lambda+\partial+\xi) $.\\
\textbf{Subcase 2:} $k(\lambda,\partial)=c\neq0$.\\
If $D(\lambda,\partial)=0$, we get $T(\lambda,\partial)=0$ by comparing the degree of $\partial $ in \eqref{lfd5}, which will cause \eqref{lfd6}  invalid. Therefore, $D(\lambda,\partial)= \partial+\gamma\lambda+\delta$, where $\gamma,\delta\in\mathbb{C}$.
It follows that $D(\lambda,\partial)=\partial+\lambda+\delta$ by comparing the coefficient of $\lambda$ in \eqref{lfd8}. Then the degree of $\partial$ in $T(\lambda,\partial)$ is smaller than 2 by comparing the degree of $\lambda$ in \eqref{lfd5}. Assume $T(\lambda,\partial)=t_1(\lambda)\partial+t_0(\lambda)$ where $t_i(\lambda)\in\mathbb{C}[\lambda]$ for $i=1, 2$. Taking it into \eqref{lfd1}, one has
\begin{align}\label{e2.3}
&(\lambda+\mu+\partial+c)(t_1(\lambda)\lambda+t_1(\lambda)\mu-t_0(\lambda))+(c+\lambda+\mu-\alpha\mu-\beta)(\partial+\lambda+\mu+\delta)\\
&=(\lambda+\mu+\partial+c)(\lambda+\partial+\mu)-(\lambda+\mu+\partial+\delta)(\mu+\partial+c)-c(t_1(-\mu-\partial)\partial+t_0(-\mu-\partial)).\nonumber
\end{align}
It follows that the degree of $\lambda$ in $t_1(\lambda)$ is equal to $0$ by comparing the coefficient of $\mu^2$ in \eqref{e2.3}. Then the degree of $\lambda$ in $t_0(\lambda)$ is smaller than $2$ by comparing the degree of $\lambda$ in \eqref{e2.3}. Assume $T(\lambda,\partial)=t_1\partial+t_2\lambda+t_3$, where $t_1,t_2,t_3\in\mathbb{C}$.  Then \eqref{e2.3} becomes
\begin{eqnarray}
\label{eqq1}&&(t_1-t_2)\lambda^2+(t_1+1-\alpha)\mu^2+(2t_1-\alpha+1-t_2)\lambda\mu+(t_1-t_2)\lambda\partial\\
&&+(t_1+1-\alpha)\mu\partial+(ct_1+c+\delta-\beta-t_3-ct_2)\lambda+(ct_1-t_3+2\delta+c-\alpha\delta-\beta-ct_2)\mu\\
&&+(-t_3+\delta+ c-\beta+ct_1-ct_2)\partial+2c\delta-\beta\delta=0.\nonumber
\end{eqnarray}
It follows that $t_2=t_1$. Taking $T(\lambda,\partial)=t_1(\partial+\lambda)+t_3 $ into
 \eqref{lfd5}, one has
\begin{align*}
(t_1\lambda+\xi-t_3)(-t_1\lambda+t_3)=t_1^2\lambda\mu-t_1t_3\mu-t_1\lambda\mu+\delta t_1\mu.
\end{align*}
Therefore, we get $t_1=0$ and $t_3^2=t_3\xi$. By \eqref{eqq1}, we have $\alpha=1$, $c+\delta=t_3+\beta$ and $(2c-\beta)\delta=0$.  Then \eqref{lfd6} becomes
\begin{align}\label{5.26}
(\beta-c)(\lambda+\mu+\partial+\xi)+(t_3-\delta)(\lambda+\mu+\partial+\beta)=-ct_3.
\end{align}
\indent If $\beta=0$, then $t_3=c=\xi\neq 0$ and $\delta=0$.

If $ \beta\neq0$, then the flag datum is equal to $(\lambda+\partial+\beta,c,\lambda+\partial+\delta,c+\delta-\beta,\lambda+\partial+c )$.

Then we consider the relationship between $c $ and $\xi $.

If $c=\xi\neq0$, then there are two cases. If $t_3\neq0$, then $\delta=\beta=2c=2t_3=2\xi$ by comparing the constant  terms  in \eqref{5.26}. If $t_3=0$ and $\delta\neq 0$, we have $\beta=2c=2\delta$. Then by comparing the constant term in \eqref{5.26} we have $c=\xi=\beta$, contradicting with our assumption.
Therefore, if $t_3=0$, then $\delta=0$ and $\beta=c$.

If $c\neq0 $ and $\xi=0$, then $t_3=0$, $\delta=0$ and $\beta=c$ in that by comparing the constant term in \eqref{5.26} we have $\beta\delta=0$ and if $ \delta\neq0$, then $\beta=2c=0$ which contradicts with our assumption.

If $c$ and $\xi$ are not equal and neither of them are $0$, then we claim that $t_3=0$. Indeed, if $t_3\neq0$, then $t_3=\xi$. By comparing the constant term in \eqref{5.26}, we get $\beta(2t_3-\delta)=0$ which means that $\delta=2t_3\neq0$ and $\beta=2c$. However $\beta-c=\delta-t_3$ implies that $c=\xi$, which contradicts with our assumption. It follows from $t_3=0$ that $\beta-c=\delta$. If $\delta\neq0$, then $\beta=\xi=2c=2\delta$.
If $\delta=0$, then $\beta=c$.

Therefore, we have the following results in this case.\par
When $c=\xi\neq0$, any flag datum is equal to one of the following forms: $(\lambda+\partial,c,\lambda+\partial,c,\lambda+\partial+c)$, $(\lambda+\partial+2c,c,\lambda+\partial+2c,c,\lambda+\partial+c)$ and
 $(\lambda+\partial+c,c,\lambda+\partial,0,\lambda+\partial+c)$.\par

When $c\neq0 $ and $\xi=0$, any flag datum is equal to $(\lambda+\partial+c,c,\lambda+\partial,0 ,\lambda+\partial) $.\par

When $\xi=2c\neq0$,  any flag datum is equal to $(\lambda+\partial+c,c,\lambda+\partial,0,\lambda+\partial+\xi )  $ and $(\lambda+\partial+2c,c,\lambda+\partial+c,0,\lambda+\partial+2c )  $.\par

When $\xi\neq c,2c$ and $c,\xi$ are not equal to $0$, any flag datum is equal to $(\lambda+\partial+c,c,\lambda+\partial,0,\lambda+\partial+\xi )  $.\\
\textbf{Subcase 3:}
In the end,  $ k(\lambda,\partial)=-\lambda+c$. By \eqref{lfd8}, we get
$$
(\lambda+\partial+c)D(\mu,\lambda+\partial)=(\mu+\partial+c)D(\lambda,\mu+\partial),
$$
which implies that $D(\lambda,\partial)$ is not equal to $\partial+\gamma\lambda+\delta$. So $D(\lambda,\partial)=0$ and then the degree of $\partial$ in $T(\lambda,\partial)$ is smaller than $2$ by comparing the degree of $\lambda$ in \eqref{lfd5}. Moreover by comparing the degree of $\partial$ in \eqref{lfd5}, we obtain $T(\lambda,\partial)=0$. It follows from \eqref{lfd6} that $ \alpha=1$ and $\xi=\beta=c$. Hence, in this case, any flag datum is equal to $(\lambda+\partial+c,-\lambda+c,0,0,\lambda+\partial+c)$.

Finally, we present our result as follows. It is obvious that the following kinds of flag datums are not equivalent to each other in each of the following cases by Theorem \ref{th4.3}.

If $\xi=c=0$ then $\mathcal{HB}^2(Q,R)$ can be described by three kinds of flag datums: $(0,0,0,0,\lambda+\partial)$,
$(\lambda+\partial, 0, \lambda+\partial, 0, \lambda+\partial)$ and $(\lambda+\partial,-\lambda,0,0,\lambda+\partial)$.\par

If $\xi=c\neq0$ then $\mathcal{HB}^2(Q,R)$ can be described by eight kinds of flag datums: $(0,0,0,0,\lambda+\partial+c)$, $(\lambda+\partial,0,\lambda+\partial+c,c,\lambda+\partial+c) $, $(\lambda+\partial,0,\lambda+\partial,0,\lambda+\partial+c )$, $(\lambda+\partial+c,0,\lambda+\partial+c,0,\lambda+\partial+c)$, $(\lambda+\partial,c,\lambda+\partial,c,\lambda+\partial+c)$, $(\lambda+\partial+2c,c,\lambda+\partial+2c,c,\lambda+\partial+c)$, $(\lambda+\partial+c,c,\lambda+\partial,0,\lambda+\partial+c)$ and $(\lambda+\partial+c,-\lambda+c,0,0,\lambda+\partial+c)$.\par

If $\xi=0$ and  $c\neq0$, then $\mathcal{HB}^2(Q,R)$ can be described by three kinds of flag datums: $(0,0,0,0,\lambda+\partial)$, $(\lambda+\partial,0,\lambda+\partial,0,\lambda+\partial)$ and $(\lambda+\partial+c,c,\lambda+\partial,0 ,\lambda+\partial) $.

If $\xi\neq0$ and  $c=0$, then $\mathcal{HB}^2(Q,R)$ can be described by three kinds of flag datums: $(0,0,0,0,\lambda+\partial+\xi)$, $(\lambda+\partial,0,\lambda+\partial+\xi,\xi,\lambda+\partial+\xi)$ and $(\lambda+\partial,0,\lambda+\partial,0,\lambda+\partial+\xi) $.

If $c=2\xi\neq0$, then $\mathcal{HB}^2(Q,R)$ can be described by five kinds of flag datums: $(0,0,0,0,\lambda+\partial+\xi)$, $(\lambda+\partial,0,\lambda+\partial+\xi,\xi,\lambda+\partial+\xi)$, $(\lambda+\partial,0,\lambda+\partial,0,\lambda+\partial+\xi) $, $(\lambda+\partial+\xi,0,\lambda+\partial+2\xi,\xi,\lambda+\partial+\xi)$ and $(\lambda+\partial+2\xi,2\xi,\lambda+\partial,0,\lambda+\partial+\xi )  $.

If $\xi=2c\neq0$, then $\mathcal{HB}^2(Q,R)$ can be described by five kinds of flag datums: $(0,0,0,0,\lambda+\partial+2c)$, $(\lambda+\partial,0,\lambda+\partial+2c,2c,\lambda+\partial+2c)$, $(\lambda+\partial,0,\lambda+\partial,0,\lambda+\partial+2c) $, $(\lambda+\partial+2c,c,\lambda+\partial+c,0,\lambda+\partial+2c )  $ and $(\lambda+\partial+c,c,\lambda+\partial,0,\lambda+\partial+2c )  $ .

If $\xi\neq2c$, $\xi\neq c$, $c\neq 2\xi$, and $\xi,c$ are not equal to 0, then  $\mathcal{HB}^2(Q,R)$ can be described by four kinds of flag datums: $(0,0,0,0,\lambda+\partial+\xi)$, $(\lambda+\partial,0,\lambda+\partial+\xi,\xi,\lambda+\partial+\xi)$, $(\lambda+\partial,0,\lambda+\partial,0,\lambda+\partial+\xi) $ and $(\lambda+\partial+c,c,\lambda+\partial,0,\lambda+\partial+\xi )  $.

\end{example}

\noindent {\bf Acknowledgments.} This research is supported by
NSFC (12171129) and
the Zhejiang Provincial Natural Science Foundation of China (LY20A010022).

\end{document}